\documentclass{amsart}

\usepackage{amsmath,amssymb,verbatim}
\usepackage{amsthm}
\usepackage{enumerate}
\usepackage{url}

\newtheorem{theorem}{Theorem}[section]
\newtheorem{fact}[theorem]{Fact}
\newtheorem{lemma}[theorem]{Lemma}
\newtheorem{corollary}[theorem]{Corollary}
\newtheorem{proposition}[theorem]{Proposition}

\theoremstyle{definition}
\newtheorem{definition}[theorem]{Definition}

\newtheorem{remark}[theorem]{Remark}

\newcommand{\abar}{\bar{a}}
\newcommand{\bbar}{\bar{b}}
\newcommand{\cbar}{\bar{c}}
\newcommand{\ubar}{\bar{u}}
\newcommand{\ybar}{\bar{y}}
\newcommand{\zbar}{\bar{z}}

\def\seq{\subseteq}
\def\nv{\text{-}}
\def\inv{^{\text{-}1}}
\def\F{\mathbb{F}}
\def\smd{\raisebox{.4pt}{\textrm{\scriptsize{~\!$\triangle$\!~}}}}
\newcommand{\claim}[1]{\hfill$\dashv_{\text{\scriptsize{Claim {#1}}}}$}

\def\Stab{\operatorname{Stab}}
\def\Aut{\operatorname{Aut}}
\def\N{\mathbb{N}}
\def\R{\mathbb{R}}
\def\Z{\mathbb{Z}}

\def\T{\mathbb{T}}
\def\cL{\mathcal{L}}
\def\cU{\mathcal{U}}
\def\cB{\mathcal{B}}
\def\cP{\mathcal{P}}
\def\cX{\mathcal{X}}
\def\cC{\mathcal{C}}

\def\cZ{\mathcal{Z}}

\title[On finite sets of small tripling or small alternation]{On finite sets of small tripling or small alternation in arbitrary groups}

\date{May 21, 2020}

\author{Gabriel Conant}

\address{Department of Pure Mathematics and Mathematical Sciences\\
University of Cambridge\\
Cambridge, CB3 0WB, UK}

\email{gconant@maths.cam.ac.uk}

\begin{document}

\begin{abstract}
We prove Bogolyubov-Ruzsa-type results for finite subsets of groups with small tripling, $|A^3|\leq O(|A|)$, or small alternation, $|AA\inv A|\leq O(|A|)$. As applications, we obtain a qualitative analog of Bogolyubov's Lemma for dense sets in arbitrary finite groups, as well as a quantitative arithmetic regularity lemma for sets of bounded VC-dimension in finite groups of bounded exponent. The latter result generalizes the abelian case, due to Alon, Fox, and Zhao, and gives a quantitative version of previous work of the author, Pillay, and Terry.  
\end{abstract}

\maketitle

\section{Introduction}\label{sec:intro}
\setcounter{theorem}{0}
\numberwithin{theorem}{section}

Freiman's Theorem (see \cite{FreiKK}, \cite{FreiFST}) is a combinatorial result in additive number theory which states that if $A$ is a finite subset of a torsion-free abelian group $G$, and $|A+A|\leq k|A|$ (i.e. $A$ has \emph{small doubling}), then $A$ is contained in an $n$-dimensional arithmetic progression of length $c|A|$, where $c$ and $n$ depend only on $k$. In \cite{Ruz94}, Ruzsa gave a new proof of this result, and a similar strategy was later used by Green and Ruzsa \cite{GrRuz} to prove a generalization of Freiman's Theorem, involving coset progressions, for arbitrary abelian groups. A key part of this work is that a set of small doubling in an abelian group can be ``Freiman-isomorphically" mapped to a dense set in a finite abelian group (see \cite[Proposition 1.2]{GrRuz}). This allows one to apply the following result, which Ruzsa \cite{Ruz94} adapted from  Bogolyubov \cite{Bog39}. For comparison to our work, we state this result in two cases.

\begin{theorem}[Bogolyubov's Lemma]\label{thm:Bogo}
Fix $r\in \Z^+$ and $\alpha\in\R^+$.
\begin{enumerate}[$(a)$]
\item \textnormal{(Bounded exponent case)} Suppose $G$ is a finite abelian group of exponent at most $r$, and $A\seq G$ is such that $|A|\geq\alpha|G|$. Then there is a subgroup $H\leq G$ such that $[G:H]\leq r^{\alpha^{\nv 2}}$ and  $H\seq 2A-2A$.
\item \textnormal{(General case)} Suppose $G$ is a finite abelian group, and $A\seq G$ is such that $|A|\geq\alpha|G|$, then there is a $(1/4,n)$-Bohr neighborhood $B$ in $G$ such that $n<\alpha^{\nv 2}$ and $B\seq 2A-2A$.
\end{enumerate}
\end{theorem}

Bohr neighborhoods (see Definition \ref{def:Bohr}) are certain kinds of well-structured sets which, in the abelian case, contain large coset progressions (preserved by Freiman isomorphism). This yields the ``\emph{Bogolyubov-Ruzsa Lemma}\footnote{This name is from Sanders \cite{SanBR}, who gives a different proof of the result yielding better bounds.} for finite abelian groups": if $G$ is abelian and $A\seq G$ is finite, with $|A+A|\leq k|A|$, then $2A-2A$ contains an $n$-dimensional coset progression of size $c|A|$, where $c$ and $n$ depend only on $k$. The conclusion of Freiman's theorem (exchanging  the arithmetic progression for one containing $A$) then follows after a little more work (see \cite[Proposition 5.1]{GrRuz}).

For an abelian group $G$, Freiman's Theorem also yields a classification of \textbf{$k$-approximate subgroups} of $G$, i.e., finite symmetric subsets $A\seq G$ such that $A+A$ can be covered by $k$ translates of $A$. Approximate subgroups of arbitrary groups have been studied by many authors, culminating in the work of Breuillard, Green, and Tao \cite{BGT}. 

The goal of the present paper is to give generalizations of Bogolyubov's Lemma to arbitrary finite groups, as well as similar statements about finite subsets of arbitrary groups whose product set growth can be controlled. For this, we focus on sets of small \emph{tripling}, which satisfy Plunnecke-Ruzsa inequalities for product sets  (as observed by Helfgott \cite{HelfGG}, see Proposition \ref{prop:Ruz}$(a)$). Motivated by the work of Hrushovski \cite{HruAG} on approximate groups, we also consider finite sets $A$ of \emph{small alternation}, i.e. $|AA\inv A|\leq k|A|$ for some fixed $k$ (see Remark \ref{rem:alttrip}).

Our main results, Theorems \ref{thm:mainbdd} and \ref{thm:maingen}, are versions of the Bogolyubov-Ruzsa Lemma for finite sets of small tripling or alternation in arbitrary groups (with some further constraints). Our first application of these results  is the following qualitative analogue of Bogolyubov's Lemma (Theorem \ref{thm:Bogo}) for arbitrary groups.

\begin{theorem}\label{thm:Bogogen}
Fix a positive integer $r$ and a positive real number $\alpha$.
\begin{enumerate}[$(a)$]
\item \textnormal{(Bounded exponent case)} Suppose $G$ is a finite group of exponent at most $r$, and $A\seq G$ is such that $|A|\geq\alpha|G|$. Then there is a normal subgroup $H\leq G$ such that $[G:H]\leq O_{r,\alpha}(1)$ and $H\seq (AA\inv)^2\cap A^2A^{\nv 2}\cap (A\inv A)^2\cap A^{\nv 2}A^2$.
\item \textnormal{(General case)} Suppose $G$ is a finite group, and $A\seq G$ is such that $|A|\geq\alpha|G|$. Then there is a normal subgroup $H\leq G$ and a $(\delta,n)$-Bohr neighborhood $B$ in $H$, such that $[G:H],\delta\inv,n\leq O_{\alpha}(1)$ and $B\seq (AA\inv)^2\cap A^2A^{\nv 2}\cap (A\inv A)^2\cap A^{\nv 2}A^2$.
\end{enumerate} 
\end{theorem}

This is proved in Section \ref{sec:results}. For later applications, and also to illustrate the use of Bohr neighborhoods in the nonabelian setting, we prove (in Section \ref{sec:results}) the following easy corollary of Theorem \ref{thm:Bogogen}$(b)$. Call a (nontrivial)  group $G$ \textbf{purely nonabelian} if no normal subgroup $H\leq G$ has a nontrivial abelian quotient (i.e. $[H,H]=H$ for all normal $H\leq G$).  The class of purely nonabelian groups contains all nonabelian simple groups,  and is closed under direct product by Goursat's Lemma. 

\begin{corollary}\label{cor:pure}
Fix a positive real number $\alpha$. Suppose $G$ is a purely nonabelian finite group and $A\seq G$ is such that $|A|\geq\alpha|G|$. Then there is a normal subgroup $H\leq G$ such that $[G:H]\leq O_\alpha(1)$ and $H\seq (AA\inv)^2\cap A^2A^{\nv 2}\cap (A\inv A)^2\cap A^{\nv 2}A^2$.
\end{corollary}

Before continuing to the next application, we state the following consequences for symmetric subsets of groups of bounded exponent (part $(a)$ follows immediately from Theorem \ref{thm:mainbdd} and part $(b)$ is a special case of Theorem \ref{thm:Bogogen}$(a)$, see Section \ref{sec:results}). 

\begin{corollary}\label{cor:sym}$~$
\begin{enumerate}[$(a)$]
\item Fix positive integers $k$ and $r$. Suppose $G$ is a group of exponent $r$ and $A\seq G$ is finite and symmetric, with $|A^3|\leq k|A|$. Then there is a subgroup $H\leq G$ such that $A$ is covered by $O_{k,r}(1)$ translates of $H$ and $H\seq A^4$.
\item Fix a positive integer $r$ and a positive real number $\alpha$. Suppose $G$ is a finite group of exponent $r$ and $A\seq G$ is symmetric, with $|A|\geq\alpha|G|$. Then there is a normal subgroup $H\leq G$, of index $O_{\alpha,r}(1)$, such that $H\seq A^4$.
\end{enumerate}
\end{corollary}

\begin{remark}\label{rem:feit}
We do not know if $H$ can also be made normal in part $(a)$ of the previous corollary.  In addition to improving Theorem \ref{thm:NIPregexp} below (see Remark \ref{rem:NIPpre}), such a result could be quite interesting, depending on the methods used. For example, together with the Feit-Thompson Theorem and the Brauer-Fowler Theorem, this strengthening of Corollary \ref{cor:sym}$(a)$ would imply that for any positive integer $r$, there are only finitely many finite simple groups of exponent $r$. This is a known fact, but its proof requires the classification of finite simple groups (e.g., \cite[Theorem 5.4]{BaGoPy}).
\end{remark}

Our final applications are in the subject of arithmetic regularity (developed by Green \cite{GreenSLAG} for abelian groups). There has been a recent interest in strengthened arithmetic regularity lemmas for subsets of groups satisfying special tameness assumptions. This was initiated by the work of Terry and Wolf \cite{TeWo} on ``stable arithmetic regularity" in $\F_p^n$, which is continued in  \cite{CPT} and \cite{TeWo2}. Arithmetic regularity in the setting of bounded VC-dimension is considered in  \cite{AFZ}, \cite{CPTNIP}, and \cite{SisNIP}. Given a group $G$ and $A\seq G$, define the \textbf{VC-dimension of $A$} to be the VC-dimension of the collection of left translates of $A$, i.e., the supremum  of all integers $d$  such that, for some $d$-element set $X\seq G$, one has $\cP(X)=\{X\cap gA:g\in G\}$. In \cite{AFZ}, Alon, Fox, and Zhao show that if $G$ is a finite \emph{abelian} group of exponent at most $r$, and $A\seq G$ has VC-dimension at most $d$, then there is a subgroup $H\leq G$ of index $(1/\epsilon)^{d+o_{r,d}(1)}$, and a subset $D\seq G$ which is a (possibly empty) union of cosets of $H$, such that $|A\smd D|\leq \epsilon|G|$. A main tool in their proof is Theorem \ref{thm:Bogo}$(a)$, and we will use Corollary \ref{cor:sym}$(a)$ to give the following generalization to arbitrary groups.

\begin{theorem}\label{thm:NIPregexp}
Fix positive integers $r$ and $d$. Suppose $G$ is a finite group of exponent at most $r$, and $A\seq G$ has VC-dimension at most $d$. Then, for any $\epsilon,\nu>0$, there is a subgroup $H$ of $G$, of index $O_{r,d,\nu}((1/\epsilon)^{d+\nu})$, which satisfies the following properties.
\begin{enumerate}[$(i)$]
\item \textnormal{(structure)} There is a set $D\seq G$, which is a union of right cosets of $H$, such that $|A\smd D|\leq \epsilon|G|$.
\item \textnormal{(regularity)} There is a set $Z\seq G$, with $|Z|<\frac{1}{2}\epsilon^{1/2}|G|$, such that for any $x\in G\backslash Z$, either $|Hx\cap A|\leq\epsilon^{1/4}|H|$ or $|Hx\backslash A|\leq\epsilon^{1/4}|H|$.
\end{enumerate}
\end{theorem}

\begin{remark}\label{rem:NIPpre}
There are several comments to be made about Theorem \ref{thm:NIPregexp}.

\begin{enumerate}[$(1)$]
\item In \cite{AFZ}, Alon, Fox, and Zhao conjecture that condition $(i)$ of Theorem \ref{thm:NIPregexp} holds for a \emph{normal} subgroup $H$ of index $\epsilon^{\nv O_{r,d}(1)}$. This would follow from the proof if one could show that Corollary \ref{cor:sym}$(a)$ holds with $H$ being normal (see Remark \ref{rem:feit}). However, it does follow from the proof  that one can replace the subgroup $H$ with the intersection of its conjugates to obtain a normal subgroup  of index $2^{\epsilon^{\nv O_{r,d}(1)}}$ satisfying conditions $(i)$ and $(ii)$ (see Remark \ref{rem:NIPreg}).

\item In \cite{CPTNIP} (joint with Pillay and Terry), we gave a version of Theorem \ref{thm:NIPregexp} in which $H$ is also normal, but without effective bounds on its index. One could instead use Corollary \ref{cor:sym}$(b)$ to deduce this, yielding a very different proof compared to what is done in \cite{CPTNIP} (see Remark \ref{rem:NIPreg}). 

\item The ``regularity" statement in condition $(ii)$ is not made explicit in \cite{AFZ}, but follows implicitly from their methods (see Lemma \ref{lem:separate}).\footnote{This was first observed by C. Terry.} 

\item The $O_{r,d, \nu}$ constant in the statement of the theorem comes from Corollary \ref{cor:sym}$(a)$ and so, unlike the abelian case, is not explicit (see Section \ref{sec:explicit}).
\end{enumerate}
\end{remark}

In Section \ref{sec:NIP} we also show that a qualitative version of Theorem \ref{thm:NIPregexp} holds for the class of purely nonabelian finite groups (see Theorem \ref{thm:NIPregpna}), which yields an interesting divergence between sets of bounded VC-dimension in nonabelian finite simple groups, compared to the abelian setting (see Corollary \ref{cor:NIPregpna}).  

We end this introduction with some discussion of our proofs. The results above involving groups of bounded exponent will be derived from Theorem \ref{thm:mainbdd}. By the work of Hrushovski \cite{HruAG} and Breuillard, Green, and Tao \cite{BGT}, approximate subgroups of groups with bounded exponent are close to genuine subgroups (see Theorem \ref{thm:BGTbdd}). Morever, in any group, finite sets  of small tripling are close to approximate subgroups by a result of Tao \cite{TaoPSE}. Together, these two facts imply a weaker version of  Theorem \ref{thm:mainbdd} (see Section \ref{sec:final}). To prove our result, we will sharpen what is essentially the first step of the work in \cite{BGT}, which is a theorem about sets of small doubling (proved independently by  Croot and Sisask \cite{CrSi} and Sanders \cite{SanBS}). Namely, if $G$ is a group, $A\seq G$ is finite, and $|A^2|\leq k|A|$, then $A^2A^{\nv 2}$ contains $S^n$, for some symmetric $S\seq G$ of size $\Omega_{k,n}(|A|)$. In Section \ref{sec:CSS}, we reprove this result using the same techniques, but for sets of small tripling or small alternation, which leads to stronger conclusions. We also work in the setting of measures (similar to Massicot and Wagner \cite{MassWa}), so that this analysis can be applied later to pseudofinite subsets of ultraproducts of groups. We then prove Theorem \ref{thm:mainbdd} in Section \ref{sec:proofbdd}.

For Theorem \ref{thm:maingen}, we will need to delve a bit deeper into the underlying methods  of \cite{BGT} and \cite{HruAG}, in particular, the ultraproduct construction. To prove the theorem, we will first prove a pseudofinite analogue, and then deduce the finitary version using an ``ultraproduct of counterexamples". To simplify this discussion, and illustrate the leverage obtained by working with pseudofinite sets, we focus on the case of symmetric sets of small tripling. In this case, the pseudofinite analog of Theorem \ref{thm:maingen} deals with a group $G$ and a \emph{pseudofinite} (symmetric) subset $A\seq G$. In other words, $G$ is an ultraproduct of groups, and $A$ is an ultraproduct of  finite (symmetric) subsets of those groups. We also assume $\langle A\rangle=A^m$ for some fixed $m$ (which holds, for example, if $A$ is an ultraproduct of uniformly dense subsets of finite groups). If $A$ has small tripling (formulated using a pseudofinite counting measure), then the Sanders-Croot-Sisask analysis from Section \ref{sec:CSS} yields a symmetric set $S$ such that $S^8\seq A^4$. Moreover, $S$ itself has small tripling (in fact it is an approximate subgroup), allowing us to iterate the process. After infinitely many iterations, we obtain a decreasing sequence of symmetric subsets of $A^4$, whose intersection is a demonstrably ``large" subgroup of $\langle A\rangle$ contained in $A^4$.

We now reach an obstacle, in that although $G$ is an ultraproduct of groups,  the subgroup constructed above need not be an ultraproduct of subgroups of those groups. In order to salvage this, we move to a  saturated elementary extension $G_*$ of $G$ (in a suitable first-order language).  We then find a normal subgroup $\Gamma$ of $\langle A_*\rangle$ of small index (where $A_*$ is the interpretation of $A$ in $G_*$), which is contained in $A_*^4$ and is an intersection of countably many definable sets. By standard facts, $\langle A_*\rangle/\Gamma$ is a compact Hausdorff group under a certain topology controlled by definable sets in $G_*$. By a result of Pillay \cite{PiRCP}, the connected component of $\langle A_*\rangle/\Gamma$ is a compact connected \emph{abelian} group, and thus an inverse limit of tori, supplying us with Bohr neighborhoods in $\langle A_*\rangle$. Finally, in order to transfer these Bohr neighborhoods  through the ultraproduct, we will use an approximation method from \cite{CPTNIP}, and a result about approximate homomorphisms from \cite{AlGlGo}. This will yield Bohr neighborhoods in the original groups, and allow us to prove Theorem \ref{thm:maingen}.

  \subsection*{Acknowledgements} I would like to thank Tim Burness, Daniel Palac\'{i}n, Anand Pillay, Caroline Terry, and Julia Wolf for helpful conversations. Thanks also to the University of Bristol School of Mathematics for their hospitality during the time this work was completed, and to the anonymous referees for valuable comments.

\section{Definitions, main theorems, and corollaries}\label{sec:results}

Before stating the main theorems, we set some notation and definitions (used throughout the paper). Let $G$ be a group. Given $n\geq 1$, let $G^{\times n}=G\times\stackrel{n}{\ldots}\times G$.  Given $X,Y\seq G$, let $XY=\{xy:x\in X,~y\in Y\}$. Set $X^0=\{1\}$ and inductively define $X^{n+1}=X^nX$. Let $X\inv=\{x\inv:x\in X\}$.  We say that $X\seq G$ is \textbf{symmetric} if $1\in X$ and $X=X\inv$. A \textbf{$Y$-translate of $X$} is a set of the form $aX$ where $a\in Y$.
Given a set $X\seq G$, we let $\langle X\rangle$ denote the subgroup of $G$ generated by $X$, and we use the notation $\bar{X}$ for the set $X\cup X\inv\cup\{1\}$. 

Let $\T^n$ denote the $n$-dimensional torus $\R/\Z\times \stackrel{n}{\ldots}\times\R/\Z$, considered as an additive group with identity $0$. Let $d_n$ denote the invariant metric on $\T^n$ induced by the product of the arclength metric on $\R/\Z$ (identified with $S^1$).

\begin{definition}\label{def:Bohr}
Given a group $G$, a positive integer $n$, and a positive real number $\delta$, we say that $B\seq G$ is a \textbf{$(\delta,n)$-Bohr neighborhood in $G$} if there is a homomorphism $\tau\colon G\to \T^n$ such that $B=B^n_{\tau,\delta}:=\{x\in G:d_n(0,\tau(x))<\delta\}$. 
\end{definition}

In the setting of abelian groups, Bohr neighborhoods are often used as replacements for subgroups in cases where few subgroups are available (e.g., in $\Z/p\Z$). In general, if $G$ is a group and $B=B^n_{\tau,\delta}$ is a $(\delta,n)$-Bohr neighborhood in $G$, then $B$ is symmetric, closed under conjugation, and contains the kernel $N$ of a homomorphism from $G$ to some $\T^n$ (so $G/N$ is abelian). While $B$ may not be closed under the group operation, one can obtain control in pairs by allowing the radius $\delta$ to vary. For instance, $B^2\seq B^n_{\tau,2\delta}$ by the triangle inequality. A more sophisticated manifestation of this idea can be found in the work of Bourgain \cite{BourgTAP}. Finally, if $G$ is finite then Bohr neighborhoods are ``large", for instance $|B^n_{\tau,\delta}|\geq \delta^n|G|$ (see \cite[Lemma 4.20]{TaoVu} or \cite[Proposition 4.5]{CPTNIP}). 

 Recall from the introduction that we are interested in finite subsets $A$ of some group $G$, which either have \emph{small tripling}, i.e., $|A^3|\leq k|A|$ for some fixed constant $k$, or have \emph{small alternation}, i.e., $|AA\inv A|\leq k|A|$ for some fixed $k$.

\begin{remark}\label{rem:alttrip}
The notion of small alternation is motivated by Hrushovski's definition of a \emph{near-subgroup} from \cite{HruAG}. Our terminology is explained by Proposition \ref{prop:Ruz}$(b)$, which shows that small alternation for a finite set $A$ in a group $G$ implies ``very small" tripling for $AA\inv$ (see  Section \ref{sec:Tao}, and especially Remark \ref{rem:HruTao}, for discussion on the relationship to approximate subgroups). Note that small tripling implies small doubling, and also small alternation due to the general Plunnecke-Ruzsa inequalties observed by Helfgott (see Proposition \ref{prop:Ruz}$(a)$). For abelian groups, small alternation clearly implies small doubling, and it is well-known that small doubling implies small tripling (see \cite{Plunn}), making  the three notions equivalent. However, in nonabelian groups, there are no general implications between small doubling and small alternation.  For example let $G$ be the free product $H\ast F_2$ where $H$ is some finite group and $F_2$ is the free group on two generators, say $a$ and $b$. Set $A=H\cup\{a\}$ and $B=aHb$. Then $A$ satisfies small doubling but not small alternation, and $B$ satisfies small alternation but not small doubling. 
\end{remark}

We now state our two main theorems, which are Bogolyubov-Ruzsa-type statements for finite sets of small alternation or small tripling. Each statement involves two crucial assumptions, the first being either small tripling or small alternation for some finite set, and the second being one of the following options: $(1)$ bounded exponent of a certain subgroup, $(2)$ bounded generation of a certain subgroup, or $(3)$ both. Altogether, this yields six statements, which we have divided into two theorems, one for the bounded exponent case and the other for the general case. The two results are proved in Sections \ref{sec:proofbdd} and \ref{sec:genproof}, respectively.

\begin{theorem}[Bounded exponent case]\label{thm:mainbdd}
Fix positive integers $k$, $m$, and $r$. Let $G$ be a group, and fix a nonempty finite subset $A\seq G$.
\begin{enumerate}[$(1)$]
\item \textnormal{(small alternation)} Suppose $|AA\inv A|\leq k|A|$ and $\langle AA\inv\rangle$ has exponent $r$. 
\begin{enumerate}[$(a)$]
 \item There is a subgroup $H\leq\langle AA\inv\rangle$ such that:
\begin{enumerate}[$(i)$]
\item $(AA\inv)^m$ is covered by $O_{k,m,r}(1)$ left cosets of $H$, and
\item $H\seq (AA\inv)^2$.
\end{enumerate}
\item Assume $\langle AA\inv\rangle=(AA\inv)^m$. Then there is a normal subgroup $H\leq \langle AA\inv\rangle$, of index $O_{k,m,r}(1)$, such that $H\seq (AA\inv)^2$.
\end{enumerate}
\item \textnormal{(small tripling)} Suppose $|A^3|\leq k|A|$ and $\langle A\rangle$ has exponent $r$.
\begin{enumerate}[$(a)$]
\item There is a subgroup $H\leq\langle A\rangle$ such that:
\begin{enumerate}[$(i)$]
\item $\bar{A}^m$ is covered by $O_{k,m,r}(1)$ left cosets of $H$, and
\item $H\seq (AA\inv)^2\cap A^2A^{\nv2}\cap (A\inv A)^2\cap A^{\nv 2} A^2$.
\end{enumerate}
\item Assume $\langle A\rangle=\bar{A}^m$. Then there is a normal subgroup $H\leq \langle AA\inv\rangle$, of index $O_{k,m,r}(1)$, such that $H\seq (AA\inv)^2$.
\end{enumerate}
\end{enumerate}
\end{theorem}

\begin{theorem}\label{thm:maingen}
Fix positive integers $k$ and $m$. Let $G$ be a group, and a fix a nonempty finite subset $A\seq G$.
\begin{enumerate}[$(1)$]
\item \textnormal{(small alternation)} Suppose $|AA\inv A|\leq k|A|$ and $\langle AA\inv \rangle =(AA\inv)^m$. Then there are:
\begin{enumerate}[\hspace{5pt}$\ast$]
\item a normal subgroup $H$ of $\langle AA\inv\rangle$, of index $O_{k,m}(1)$, and
\item a $(\delta,n)$-Bohr neighborhood $B$ in $H$, with $\delta\inv,n\leq O_{k,m}(1)$,
\end{enumerate}
such that $B\seq (AA\inv)^2$. Moreover, if $\langle AA\inv\rangle$ is abelian, then we may assume $H=\langle AA\inv\rangle$.
\item \textnormal{(small tripling)} Suppose $|A^3|\leq k|A|$ and $\langle A\rangle=\bar{A}^m$. Then there are:
\begin{enumerate}[\hspace{5pt}$\ast$]
\item a normal subgroup $H$ of $\langle A\rangle$, of index $O_{k,m}(1)$, and
\item a $(\delta,n)$-Bohr neighborhood $B$ in $H$, with $\delta\inv,n\leq O_{k,m}(1)$,
\end{enumerate}
such that $B\seq (AA\inv)^2\cap A^2A^{\nv2}\cap (A\inv A)^2\cap A^{\nv 2} A^2$. Moreover, if $\langle A\rangle$ is abelian then we may assume $H=\langle A\rangle$.
\end{enumerate}
\end{theorem}

Since the work of Breuillard, Green, and Tao \cite{BGT} on approximate groups makes several appearances in this paper, we take a moment to reconcile their work with the theorems above. First, Theorem \ref{thm:mainbdd} strengthens the main result from \cite{BGT} on approximate subgroups of groups of bounded exponent (see Theorem \ref{thm:BGTbdd}), in that we have replaced approximate subgroups with sets of small alternation or small tripling. As discussed in the introduction, this improvement is obtained by modifying the first step of the work in \cite{BGT}, and then applying their structure theorem. (See also Section \ref{sec:Tao}, where discuss further consequences of our work for the structural results on approximate subgroups from \cite{BGT}.)

To compare Theorem \ref{thm:maingen} to \cite{BGT}, we first quote one of their main results.

\begin{theorem}\textnormal{\cite[Theorem 1.6]{BGT}}\label{thm:BGTmain}
Fix a positive integer $k$. Suppose $G$ is a group and $A\seq G$ is a finite $k$-approximate subgroup of $G$. Then there is a subgroup $H$ of $G$ and a finite normal subgroup $N$ of $H$ with the following properties:
\begin{enumerate}[$(i)$]
\item $A$ is covered by $O_k(1)$ left translates of $H$;
\item $H/N$ is nilpotent and finitely generated of rank and step  $O_k(1)$;
\item $A^4$ contains $N$ and a generating set for $H$.
\end{enumerate}
\end{theorem}

For comparison, in both parts of Theorem \ref{thm:maingen},  the Bohr neighborhood $B$ contains the kernel $N$ of a homomorphism from $H$ to $\T^n$. Thus $H/N$ is a finite abelian group, which can be generated by $n\leq O_{k,m}(1)$ elements. Moreover, since $|B|\geq \Omega_{k,m}(|H|)$, we could replace $H$ by the  the subgroup generated by $B$, and have that $B$ contains a generating set of $H$ (although possibly losing normality of $H$). Altogether, Theorem \ref{thm:maingen} can be seen as an analog of Theorem \ref{thm:BGTmain}, where we obtain stronger conclusions for sets of small alternation or small tripling, under the extra ``bounded generation" assumption coming from the parameter $m$. As with Theorem \ref{thm:mainbdd}, our proof of Theorem \ref{thm:maingen} relies on \cite{BGT}, although this time implicitly via a result of Pillay \cite{PiRCP} used in Proposition \ref{prop:Bohr}. However, this dependence on  \cite{BGT} could be avoided by using  a generalization of Pillay's result due to Nikolov, Schneider, and Thom \cite{NST}.

\begin{remark}
Recall that  the  ``Bogolyubov-Ruzsa Lemma" for \emph{abelian} groups, discussed after Theorem \ref{thm:Bogo}, does not involve a ``bounded generation" parameter $m$ like in Theorem \ref{thm:maingen}. However, similar to Freiman's Theorem,  this result for abelian groups reduces to Theorem \ref{thm:Bogo}$(b)$, using the fact that Bohr neighborhoods in abelian groups can be approximated by coset progressions (see \cite[Lemma 4.22]{TaoVu}), and that finite sets of small doubling in abelian groups have ``good models" as dense sets in finite abelian groups (see \cite[Proposition 2.1]{GrRuz}). Thus, for the sake of completeness, we will explain in Remark \ref{rem:abelianBohr} how to obtain $G=H$ in Theorem \ref{thm:Bogogen}$(b)$ when $G$ is abelian.  
\end{remark}

The rest of this section is devoted to proving the theorems and corollaries in Section \ref{sec:intro} (except for Theorem \ref{thm:NIPregexp}, which is proved in Section \ref{sec:NIP}). We first consider Corollary \ref{cor:sym} since it is immediate from the theorems above.

\begin{proof}[Proof of Corollary \ref{cor:sym}]
Part $(a)$ is immediate from Theorem \ref{thm:mainbdd}. Part $(b)$ is immediate from Theorem \ref{thm:Bogogen}$(a)$.
\end{proof}

\begin{proof}[Proof of Theorem \ref{thm:Bogogen}]
Part $(a)$. Fix a positive integer $r$ and a positive real number $\alpha$. Suppose $G$ is a finite group of exponent $r$ and $A\seq G$ is such that $|A|\geq\alpha|G|$. It is straightforward to show that $\langle A\rangle=\bar{A}^m$ for some $m\leq\lceil 3\alpha+1\rceil$ (see, e.g., \cite[Remark 4]{MassWa}). So we can apply Theorem \ref{thm:mainbdd}$(2)(b)$, with $k=\lceil\alpha\inv\rceil$ and $m=\lceil 3\alpha+1\rceil$, to obtain a  subgroup $K\leq \langle A\rangle$, of index $n=n(\alpha,r)$, such that $K\seq (AA\inv)^2\cap A^2A^{\nv2}\cap (A\inv A)^2\cap A^{\nv 2} A^2$. Note also that $[G:\langle A\rangle]\leq\lceil\alpha\inv\rceil$, and so $[G:K]\leq n\lceil\alpha\inv\rceil$. Now, if $H=\bigcap_{g\in G}gKg\inv$, then $[G:H]\leq [G:K]!\leq O_{\alpha,r}(1)$, $H$ is normal in $G$, and $H\seq (AA\inv)^2\cap A^2A^{\nv2}\cap (A\inv A)^2\cap A^{\nv 2} A^2$.

Part $(b)$. Fix a positive real number $\alpha$. Suppose $G$ is a finite group and $A\seq G$ is such that $|A|\geq\alpha|G|$. In analogy to part $(a)$, Theorem \ref{thm:maingen}$(2)$ provides a subgroup $K\leq\langle A\rangle$ and a Bohr neighborhood $B^n_{\tau,\delta}\seq K$ such that  $[G:K]$, $\delta\inv$, and $n$ are bounded above in terms of $\alpha$, and $B^n_{\tau,\delta}\seq (AA\inv)^2\cap A^2A^{\nv2}\cap (A\inv A)^2\cap A^{\nv 2} A^2$. If $H=\bigcap_{g\in G}gKg\inv$ and  $B=B^n_{\tau,\delta}\cap H$, then $H$ is normal in $G$, $[G:H]\leq O_{\alpha}(1)$, and $B=B^n_{\tau{\upharpoonright}H,\delta}$. So $B$ and $H$ are as desired.
\end{proof}

\begin{remark}
It is worth pointing out that part $(a)$ of Theorem \ref{thm:Bogogen} also follows easily from part $(b)$, due to the fact that a $(\delta,n)$-Bohr neighborhood in a group of exponent $r>\delta\inv$ is a subgroup. We leave details to the reader. In a similar way, parts $(1b)$ and $(2b)$ of Theorem \ref{thm:mainbdd} follows from parts $(1)$ and $(2)$ of Theorem \ref{thm:maingen}, respectively. On the other hand, the proof given below of Theorem \ref{thm:mainbdd} is more direct, and does not require the model theoretic methods employed here, nor the work from \cite{CPTNIP} on approximate Bohr neighborhoods.
\end{remark}

\begin{proof}[Proof of Corollary \ref{cor:pure}]
Fix $\alpha>0$. Suppose $G$ is a purely nonabelian finite group and $A\seq G$ is such that $|A|\geq\alpha|G|$.   By Theorem \ref{thm:Bogogen}$(b)$, there is a normal subgroup $H\leq G$ and a $(\delta,n)$-Bohr neighborhood $B\seq H$, such that $[G:H]\leq O_\alpha(1)$ and $B\seq (AA\inv)^2\cap (A\inv A)^2\cap A^2A^{\nv 2}\cap A^{\nv 2}A^2$. Since $B$ contains $\ker(\tau)$ for some homomorphism $\tau\colon H\to\T^n$, and $G$ is purely nonabelian, we must have $B=H$.
\end{proof}

\begin{remark}\label{rem:simple}
Corollary \ref{cor:pure} implies that for any $\alpha>0$, if $G$ is a nonabelian finite simple group of size at least $\Omega_\alpha(1)$, and $A\seq G$ is such that $|A|\geq\alpha |G|$, then $G=(AA\inv)^2= (A\inv A)^2= A^2A^{\nv 2}= A^{\nv 2}A^2$. Applied to the case of alternating groups $A_n$, this implies that the least upper bound on the index of $H$ in Theorem \ref{thm:Bogogen}$(b)$ must be greater than $\frac{1}{2}\lfloor \alpha\inv\rfloor!$ (at least for $\alpha\leq \frac{1}{5}$). However, it should be noted that stronger results about dense sets in nonabelian finite simple groups are already known. In particular, if $G$ is a nonabelian finite simple group with $\log|G|\geq\Omega(\alpha^{\nv 6})$, and $A,B,C\seq G$ are such that $|A|,|B|,|C|\geq\alpha|G|$, then $G=ABC$.\footnote{By \cite{GowQRG}, the implied constant in $\Omega(\alpha^{\nv 6})$ is no more than $25^{\log(25)}$. Using the classification of finite simple groups, the overall bound can be improved to $|G|>(\lceil\alpha^{\nv 3}\rceil+1)!$ (see \cite{CollJCLG}).} This follows from work of Gowers \cite{GowQRG} on quasirandom groups (as observed by Nikolov and Pyber \cite{NikPyDJ}; see see \cite[Corollary 1]{NikPyDJ}, \cite[Theorem 3.3]{GowQRG}, and \cite[Theorem 4.7]{GowQRG}). Similar results are shown by Hrushovski in \cite{HruAG} (e.g. \cite[Corollary 1.4]{HruAG}). 
\end{remark}

\section{Ultraproducts of groups}\label{sec:G}

In this section we review the ultraproduct construction in the case of groups. The reader only interested in Theorem \ref{thm:maingen} (the bounded exponent case) can skip this section. Throughout this section, let $(G_s)_{s\in\N}$ be a fixed sequence of groups, and fix a nonprincipal ultrafilter $\cU$ on $\N$. Let $G=\prod_{\cU}G_s$ be the ultraproduct of the sequence $(G_s)_{s\in\N}$ with respect to $\cU$. Explicitly, $G=(\prod_s G_{s\in\N})/\!\!\sim$, where $(a_s)\sim(b_s)$ if and only if $\{s:a_s=b_s\}\in\cU$. Recall that $G$ is a group under the (well-defined) operation $[(a_s)]\cdot [(b_s)]=[(a_s\cdot b_s)]$.  A subset $X\seq G$ is \textbf{internal} if there is a sequence $(X_s)_{s\in\N}$, with $X_s\seq G_s$, such that $X=\prod_{\cU}X_s:=(\prod_{s\in \N}X_s)/\!\!\sim$. The collection of internal subsets of $G$ forms a Boolean algebra. 

We also assume that $G$ is infinite, i.e., $\{s\in \N:|G_s|>n\}\in\cU$ for all $n\in\N$. As a result, we obtain the following saturation property of $G$. 

\begin{fact}[Keisler \cite{Keis64}]\label{fact:Keisler}
Suppose $(X_i)_{i=0}^\infty$ is a sequence of internal subsets of $G^{\times n}$ such that $\bigcap_{i=0}^k X_i\neq\emptyset$ for all $k\in\N$. Then $\bigcap_{i=0}^\infty X_i\neq\emptyset$.
\end{fact}

Finally, we fix a distinguished internal set $A\seq G$ (so $A=\prod_{\cU}A_s$ for some $A_s\seq G_s$), and we assume that $A$ is nonempty and \emph{pseudofinite} (i.e., $A_s$ is nonempty and finite for all $s\in\N$). With $A$ fixed, we define the \textbf{$|A|$-normalized pseudofinite counting measure} $\mu$ on the Boolean algebra of internal subsets of $G$. Specifically, given an internal set $X=\prod_{\cU}X_s$, define
\[
\mu(X)=\lim_{\cU}\frac{|X_s|}{|A_s|}\in\R_{\geq 0}\cup\{\infty\},
\]
(where $\lim_{\cU} x_s=y$ if and only if, for all $\epsilon>0$, $\{s:|x_s-y|<\epsilon\}\in\cU$). Note that $\mu$ is a left and right invariant finitely additive measure on the internal subsets of $G$.

Properties of finite subsets of groups such as small alternation or small tripling can be formulated using $\mu$. For example, $\mu(A^3)<\infty$ if and only if for some fixed $k>0$, $\{s:|A_s^3|\leq k|A_s|\}\in\cU$. The fact that $\mu$ is controlled by discrete counting measures allows us to transfer the following Plunnecke-Ruzsa inequalities to $G$. 

\begin{proposition}\label{prop:Ruz}
Fix an internal set $X\seq G$, with $0<\mu(X)<\infty$.
\begin{enumerate}[$(a)$]
\item Suppose $\mu(X^3)\leq k\mu(X)$ for some $k>0$. Then, for any $n\geq 1$ and $\epsilon_1,\ldots,\epsilon_n\in\{\nv1,1\}$, $\mu(X^{\epsilon_1}\cdot\ldots\cdot X^{\epsilon_n})\leq k^{O_n(1)}\mu(X)$.
\item Suppose $\mu(XX\inv X)\leq k\mu(X)$ for some $k>0$. Then $\mu((XX\inv)^n)\leq k^{O_n(1)}\mu(X)$ for any $n\geq 1$.
\end{enumerate}
\end{proposition}
\begin{proof}
It suffices to fix $s\in\N$ and prove the claims for $G_s$ with $\mu$ replaced by the usual counting measure. In this setting, part $(a)$ is precisely the ``discrete case" of \cite[Lemma 3.4]{TaoPSE} (first observed by Helfgott \cite[Lemma 2.2]{HelfGG}). So we only need to show part $(b)$. The proof is similar to that of \cite[Lemma 3.4]{TaoPSE} and relies on the triangle inequality for Ruzsa distance. In particular, given nonempty finite $X,Y\seq G_s$, the Ruzsa distance is between $X$ and $Y$ is defined as
\[
d(X,Y)=\log\left(\frac{|XY\inv|}{|X|^{1/2}|Y|^{1/2}}\right).
\]
Then, for nonempty finite $X,Y,Z\seq G_s$, we have $d(X,Z)\leq d(X,Y)+d(Y,Z)$ (this is due to Ruzsa \cite{Ruz} in the commutative setting; see also \cite[Lemma 3.2]{TaoPSE}). 

Now fix a nonempty finite set $X\seq G_s$, and assume $|XX\inv X|\leq k|X|$. By part $(a)$ (in $G_s$) it is enough to show $|(XX\inv)^3|\leq k^{O(1)}|X|$. For this, first note that $d(XX\inv,X\inv)\leq\log k$. So $d(XX\inv,XX\inv)\leq \log k^2$ by the triangle inequality, and thus $|(XX\inv)^2|\leq k^2|XX\inv|$.
Then $d(X X\inv X,X)\leq \log k^2$. By the triangle inequality, $d(X X\inv X,X X\inv X)\leq \log k^4$, and thus $|(XX\inv)^3|\leq k^4|XX\inv X|\leq k^5|X|$.
\end{proof}

\section{Sanders-Croot-Sisask Analysis}\label{sec:CSS}

In this section, we prove Lemma \ref{lem:MWKP}, which is the main technical lemma of the paper. It is essentially a  modification of a result of Croot-Sisask \cite{CrSi} and Sanders \cite{SanBS}, which was later adapted by Breuillard, Green, and Tao \cite[Section 5]{BGT} for their results on the structure of approximate groups. In the model-theoretic setting, these same techniques were used by Massicot and Wagner \cite{MassWa} in their work on ``definably amenable" approximate groups, and also by Krupi{\'n}ski and Pillay \cite{KrPiAG}. Part $(a)$ of Lemma \ref{lem:MWKP}, which deals with sets of small alternation, is similar to some of Hrushovski's work with near-subgroups, especially \cite[Corollary 3.11]{HruAG}. Our proof follows Sanders \cite{SanBS} (as do  \cite{BGT}, \cite{KrPiAG}, and \cite{MassWa}),  and makes the modifications necessary to work with sets of small tripling or small alternation, and also to account for working with internal sets in the case of ultraproducts.

In this section, we work with a fixed group $G$, a fixed subset $A\seq G$, and a finitely additive measure $\mu$, defined on a certain Boolean algebra of subsets of $G$ and taking values in $\R_{\geq 0}\cup\{\infty\}$. While one could formulate a precise axiomatic framework to allow for a more general setting, it will suffice for our purposes to further assume that  one of the following two cases holds.

\hspace{5pt}

\noindent\textbf{Discrete case:} $A$ is a nonempty finite subset of $G$ and $\mu$ is the $|A|$-normalized counting measure: $\mu(X)=|X|/|A|$, defined for any $X\seq G$. 

\hspace{5pt}

\noindent\textbf{Pseudofinite case:} $G$, $A$, and $\mu$ are as in Section \ref{sec:G}.

\hspace{5pt}

The reader only interested in Theorem \ref{thm:maingen} can assume the discrete case and ignore the pseudofinite case. We call a set $X\seq G$ \textbf{measurable} if $\mu(X)$ is defined. 
Note that Proposition \ref{prop:Ruz} makes sense in the discrete case if we remove the word ``internal", and the statement remains true (this is what was shown in the proof). 

\begin{lemma}\label{lem:MWKP}
Fix $m,n\geq 1$ and a measurable set $X\seq G$, with $0<\mu(X)<\infty$.
\begin{enumerate}[$(a)$]
\item Suppose $\mu(XX\inv X)\leq k\mu(X)$ for some $k\geq 1$. Then there is a measurable symmetric set $Y\seq G$ such that $Y^n\seq (X X\inv)^2$ and $(XX\inv)^m$ is covered by $O_{k,m,n}(1)$ $(XX\inv)^m$-translates of $Y$.
\item Suppose $\mu(X^3)\leq k\mu(X)$ for some $k\geq 1$. Then there is a measurable symmetric set $Y\seq G$ such that $Y^n\seq (XX\inv)^2\cap X^2X^{\nv 2}\cap (X\inv X)^2\cap X^{\nv 2}X^2$ and $\bar{X}^{m}$ is covered by $O_{k,m,n}(1)$ $\bar{X}^{m}$-translates of $Y$.
\end{enumerate}
\end{lemma}
\begin{proof}
We will prove the two statements in parallel, as the arguments are similar.
 Let $X\seq G$ be a fixed internal set, with $0<\mu(X)<\infty$. Fix an integer $k\geq 1$. By ``case $(a)$", we mean the assumption that $\mu(XX\inv X)\leq k\mu(X)$; and by ``case $(b)$'', we mean the assumption that $\mu(X^3)\leq k\mu(X)$. 
 
 Before starting the argument, we first simplify case $(b)$. Define
 \[
 \Pi_1(X)=(XX\inv)^2,~\Pi_2(X)=X^2X^{\nv 2},~\Pi_3(X)=(X\inv X)^2,\text{ and }\Pi_4(X)=X^{\nv 2}X^2.
 \]
We claim that it suffices to find, for each individual $c\in\{1,2,3,4\}$, a set $Y_c$ as described but only with $Y_c^n\seq \Pi_c(X)$. To see this, we apply some elementary tools from \cite{BGT} (which transfer to the pseudofinite setting).  So fix $m,n\geq 1$ and suppose that, for $c\in\{1,2,3,4\}$, we have measurable symmetric $Y_c\seq G$ such that $Y_c^{4n}\seq\Pi_c(X)$ and $\bar{X}^m$ is covered by $O_{k,m,n}(1)$ $\bar{X}^m$-translates of $Y_c$. Let $S=\bar{X}^{\max\{m,2\}}$. Then $S$ is an $O_{k,m}(1)$-approximate subgroup by Proposition \ref{prop:Ruz}$(a)$ and \cite[Corollary 5.2]{BGT}. Setting $Y_*=\bigcap_{c=1}^4 Y^2_c$, we have $\mu(\bar{X}^m)\leq\mu(S)\leq O_{k,m,n}(\mu(Y_*))$ by \cite[Corollary 5.9]{BGT}. By \cite[Lemma 5.1]{BGT}, there is a finite set $F\seq \bar{X}^m$ such that $|F|\leq\mu(\bar{X}^mY_*)/\mu(Y_*)$ and $\bar{X}^m\seq FY_*^2$. Since $Y_*\seq \bar{X}^2$, we have $\mu(\bar{X}^mY_*)\leq O_{k,m}(\mu(\bar{X}^m))$ by Proposition \ref{prop:Ruz}$(a)$, which implies $|F|\leq O_{k,m,n}(1)$. So if we set $Y=Y_*^2$, then $Y$ is a measurable symmetric set, $\bar{X}^m$ is covered by $O_{k,m,n}(1)$ $\bar{X}^m$-translates of $Y$, and $Y^n= Y_*^{2n}\seq \bigcap_{c=1}^4 Y^{4n}_c\seq \bigcap_{c=1}^4\Pi_c(X)$, as desired. 

So now in case $(b)$, we fix $c\in\{1,2,3,4\}$ and find $Y=Y_c$  with $Y^n\seq\Pi_c(X)$.  Since $\Pi_1(X\inv)=\Pi_3(X)$ and $\Pi_2(X\inv)=\Pi_4(X)$, it suffices to assume $c\in\{1,2\}$. Set 
\[
V=\begin{cases}
(XX\inv)^m & \text{in case $(a)$}\\
\bar{X}^{m} & \text{in case $(b)$}
\end{cases}
\makebox[.4in]{and}
Z=\begin{cases}
X\inv & \text{in case $(a)$, or case $(b)$ with $c=1$}\\
X & \text{in case $(b)$ with $c=2$.}
\end{cases}
\]
Note that $V$ is symmetric. We now closely follow Sanders \cite{SanBS}. For $t\in (0,1]$, define
\[
\cB_t=\{B\seq X:\text{$B$ is internal and $\mu(B)\geq t\mu(X)$}\}.
\]
Then $X\in \cB_t$ for all $t\in(0,1]$, and so we may define a function $f\colon (0,1]\to[1,\infty)$ such that $f(t) =\inf\{\mu(BZ)/\mu(X):B\in\cB_t\}$. 

By Proposition \ref{prop:Ruz}, we may fix $\ell\geq 1$ such that $\ell\leq k^{O_m(1)}$ and $\mu(VX)\leq \ell\mu(X)$ (in case $(a)$, use $\mu((XX\inv)^mX)\leq \mu((XX\inv)^{m+1})$). By \cite[Lemma 11]{MassWa} (taken from  \cite{SanBS}), we may choose $t\in (0,1]$ such that $t\inv\leq O_{k,m,n}(1)$ and $f(t^2/2\ell)\geq ((2n-1)/2n)f(t)$. Choose $B\in\cB_t$ such that $\mu(BZ)/\mu(X)\leq ((2n+1)/2n)f(t)$.

Define $Y_*=\{g\in V^2:\mu(gB\cap B)\geq (t^2/2\ell)\mu(X)\}$,  and note that $1\in Y_*$, since $B\in\cB_t$ and $t>t^2/2\ell$. 

\vspace{5pt}

\noindent\emph{Claim 1}: $V$ is covered by $O_{k,m,n}(1)$ $V$-translates of $Y_*$.

\noindent\emph{Proof}: Let $w=\lfloor 2\ell/t\rfloor$ and note that $w\leq O_{k,m,n}(1)$. Suppose, for a contradiction, that $V$ is not covered by $w$ $V$-translates of $Y_*$. Then we may construct a sequence $(g_i)_{i=0}^w$ from $V$ such that, for all $i\leq w$, $g_i\not\in\bigcup_{j<i}g_j Y_*$. For any $0\leq i<j\leq w$, $g_i\inv g_j\in V^2\backslash Y_*$, and so we have $\mu(g_iB\cap g_jB)<(t^2/2\ell)\mu(X)$. We also have $g_iB\seq VX$ for any $0\leq i\leq w$. Now we obtain a contradiction:
\begin{align*}
\ell\mu(X) &\geq \mu(VX) \geq \mu\left(\bigcup_{i\leq w}g_i B\right) \geq (w+1)\mu(B)-\sum_{i<j\leq w}\mu(g_iB\cap g_jB)\\
 &> (w+1)t\mu(X)-\frac{w(w+1)t^2\mu(X)}{4\ell}= (w+1)\left(1-\frac{wt}{4\ell}\right)t\mu(X)\\
  &> \ell\mu(X),
\end{align*}
where the last inequality uses $w\leq 2\ell/t< w+1$.\claim{1}

Set
\[
W=\begin{cases}
(XX\inv)^2 & \text{in case $(a)$, or case $(b)$ with $c=1$}\\
X^2X^{\nv 2} & \text{in case $(b)$ with $c=2$.}
\end{cases}
\]

\noindent\emph{Claim 2}: $Y^n_*\seq W$. 

\noindent\emph{Proof}: We first show that $\mu(gBZ\smd BZ)<2\mu(BZ)/n$ for any $g\in Y_*$. To see this, note that if $g\in Y_*$ then $gB\cap B\in\cB_{t^2/2\ell}$, and so
\[
\mu(gBZ\cap BZ)\geq\mu((gB\cap B)Z)\geq f(t^2/2\ell)\mu(X)\geq \frac{2n-1}{2n}f(t)\mu(X)\geq\frac{2n-1}{2n+1}\mu(BZ).
\]
So, for any $g\in Y_*$, 
\[
\mu(gBZ\smd BZ)=2\mu(BZ)-2\mu(gBZ\cap BZ)\leq \frac{4}{2n+1}\mu(BZ)<\frac{2}{n}\mu(BZ).
\]
Now fix $g_1,\ldots,g_n\in Y_*$ and, for $0\leq i\leq n$, let $h_i=\prod_{j\leq i} g_j$ (so $h_0=1$). Then
\begin{align*}
\mu(h_nBZ\smd BZ) \leq \mu\left(\bigcup_{i=0}^{n-1} h_i(g_{i+1}BZ\smd BZ)\right)\leq \sum_{i=0}^{n-1}\mu(g_{i+1}BZ\smd BZ)<2\mu(BZ).
\end{align*}
It follows that $h_nBZ\cap BZ\neq\emptyset$, which implies $h_n\in BZZ\inv B\inv\seq W$.\claim{2}

\hspace{5pt}

Now, in the discrete case, we may take $Y=Y_*$ and the proof is finished. In the pseudofinite case, we must address the fact that $Y_*$ may not be internal. So suppose  we are in the pseudofinite case. 

\hspace{5pt}

\noindent\emph{Claim 3}: $Y_*=\bigcap_{i=0}^\infty Y_i$ where, for each $i\in\N$, $Y_i$ is symmetric and internal, and contains $Y_{i+1}$.

\noindent\emph{Proof}: Let $\beta=(t^2/2\ell)\mu(X)$, and so $Y_*=\{g\in V^2:\mu(gB\cap B)\geq \beta\}$. By assumption, $X$ and $B$ are internal and so we may choose sets $X_s,B_s\seq G_s$, for $s\in\N$, such that $X=\prod_\cU X_s$ and $B=\prod_\cU B_s$. Given $s\in\N$, set
\[
V_s=\begin{cases}
(X_sX_s\inv)^n & \text{in case $(a)$,}\\
\bar{X}_s^{n} & \text{in case $(b)$.}
\end{cases}
\]
Note that each $V_s$ is symmetric, and $V=\prod_{\cU}V_s$. Given $i\in\N$ and $s\in\N$, define 
\[
Y_{i,s}=\left\{g\in V_s^2:\frac{|gB_s\cap B_s|}{|G_s|}>\textstyle\beta-\frac{1}{i+1}\right\}.
\]
Note that $Y_{i,s}\inv = Y_{i,s}$ for all $i,s\in\N$.  Given $i\in\N$, let $Y_i=\prod_\cU Y_{i,s}$. Then, for any $i\in\N$, $Y\inv_i=Y_i$ and $Y_{i+1}\seq Y_i$. Moreover, $Y_*=\bigcap_{i=0}^\infty Y_i$. \claim{3}

\hspace{5pt}

Fix $(Y_i)_{i=0}^\infty$ as in Claim 3. To finish the proof of the lemma in the pseudofinite case, it suffices to show that $Y^n_i\seq W$ for some $i\in\N$.  Toward this end, we first show $Y_*^n=\bigcap_{i=0}^\infty Y^n_i$. We clearly have $Y_*^n\seq \bigcap_{i=0}^\infty Y^n_i$. For the other direction, fix $a\in \bigcap_{i=0}^\infty Y_i^n$. For $i\in\N$, define
\[
D_i=\{(g_1,\ldots,g_n)\in G^{\times n}:g_j\in Y_i\text{ for $1\leq j\leq i$, and $a=g_1\cdot\ldots \cdot g_n$}\}.
\]
Then, for all $i\in\N$, $D_i$ is nonempty, internal, and $D_{i+1}\seq D_i$. By Fact \ref{fact:Keisler}, there is $(g_1,\ldots,g_n)\in \bigcap_{i=0}^\infty D_i$, and so $a=g_1\cdot \ldots \cdot g_n\in Y^n_*$. 

Finally, since $\bigcap_{i=0}^{\infty} Y^n_i=Y_*^n\seq W$, it follows from Fact \ref{fact:Keisler} that $Y^n_i\seq W$ for some $i\in\N$. 
\end{proof}

At this point, we have all necessary tools to proceed with the proof of Theorem \ref{thm:mainbdd} (see Section \ref{sec:proofbdd}). For Theorem \ref{thm:maingen}, we will 
need following corollary of the previous lemma, which is only meaningful in the pseudofinite case.

\begin{corollary}\label{cor:MWKP}$~$
\begin{enumerate}[$(a)$]
\item Suppose $\mu(AA\inv A)\leq k<\infty$. Then there is a sequence $(X_n)_{n=0}^\infty$ of symmetric, internal subsets of $G$ such that $X_0\seq (AA\inv)^2$ and, for any $n\in\N$, $X^2_{n+1}\seq X_n$ and $(AA\inv)^4$ is covered by $O_{k,n}(1)$ $\langle AA\inv\rangle$-translates of $X_n$.
\item Suppose $\mu(A^3)\leq k<\infty$. Then there is a sequence $(X_n)_{n=0}^\infty$ of symmetric, internal subsets of $G$ such that $X_0\seq (AA\inv)^2\cap A^2A^{\nv 2}\cap (A\inv A)^2\cap A^{\nv 2}A^2$ and, for any $n\in\N$, $X^2_{n+1}\seq X_n$ and $\bar{A}^8$ is covered by $O_{k,n}(1)$ $\langle A\rangle$-translates of $X_n$.
\end{enumerate}
\end{corollary}
\begin{proof}
As in Lemma \ref{lem:MWKP}, we prove the two statements in parallel. Set
\[
(V,W,\Sigma)=\begin{cases}
((AA\inv)^4,(AA\inv)^2,\langle AA\inv\rangle) & \text{in case $(a)$,}\\
(\bar{A}^8,(AA\inv)^2\cap A^2A^{\nv 2}\cap (A\inv A)^2\cap A^{\nv 2}A^2,\langle A\rangle) & \text{in case $(b)$.}
\end{cases}
\]

We construct a sequence $(Y_n)_{n=0}^\infty$ of symmetric, internal subsets of $G$, such that $Y_0^8\seq W$ and, for all $n\in\N$, $Y_{n+1}^8\seq Y_n^4$ and $V$ is covered by $O_{K,n}(1)$ $\Sigma$-translates of $Y_n$. Given this, the result follows by setting $X_n=Y^4_n$.

By Lemma \ref{lem:MWKP}, there is a symmetric, internal set $Y_0\seq G$ such that $Y_0^8\seq W$ and $V$ is covered by $O_k(1)$ $V$-translates of $Y_0$. Suppose we have constructed $Y_0,\ldots,Y_n$ satisfying the desired properties. Note that $\mu(W)<\infty$ by Proposition \ref{prop:Ruz}, and so $\mu(Y_n)<\infty$ since $Y_n\seq Y^8_n\seq W$. Since $V$ is covered by $O_{k,n}(1)$ translates of $Y_n$, we have $0<\mu(V)\leq O_{k,n}(1)\mu(Y_n)$, and so $\mu(Y_n)>0$. Since $Y^3_n\seq W$, we also have $\mu(Y^3_n)\leq O_{k,n}(1)\mu(Y_n)$. By Lemma \ref{lem:MWKP}$(b)$, there is a symmetric, internal $Y_{n+1}\seq G$ such that $Y_{n+1}^8\seq Y_n^4$ and $Y_n^8$ is covered by $O_{k,n+1}(1)$ $Y_n^8$-translates of $Y_{n+1}$. Since $Y_n\seq Y_n^8\seq W$, it follows that $Y_n$ is covered by $O_{k,n+1}(1)$ $W$-translates of $Y_{n+1}$. Since $V$ is covered by $O_{k,n}(1)$ $\Sigma$-translates of $Y_n$, it follows that $V$ is covered by $O_{k,n+1}(1)$ $\Sigma$-translates of $Y_{n+1}$.
\end{proof}

\section{Proof of Theorem \ref{thm:mainbdd}}\label{sec:proofbdd}

The following theorem is  \cite[Theorem 6.15]{BGT}. It can also be deduced from \cite[Corollary 4.18]{HruAG}.

\begin{theorem}[\cite{BGT}]\label{thm:BGTbdd}
Let $G$ be a group of exponent $r$, and suppose $X\seq G$ is a $k$-approximate subgroup. Then $X^4$ contains a subgroup $H\leq G$ such that $X$ is covered by $O_{k,r}(1)$ left cosets of $H$.
\end{theorem}

We now give the proof of Theorem \ref{thm:mainbdd}.

\begin{proof}[Proof of Theorem \ref{thm:mainbdd}]
We prove parts $(1)$ and $(2)$ of the theorem in two parallel cases.
Fix positive integers $k$, $m$, and $r$. Let $G$ be a group and fix $A\seq G$ nonempty and finite, with $|AA\inv A|\leq k|A|$ in case $(1)$ and $|A^3|\leq k|A|$ in case $(2)$.
Set
\[
(V,W,\Sigma)=\begin{cases}
(AA\inv,(AA\inv)^2,\langle AA\inv\rangle) & \text{in case $(1)$,}\\
(\bar{A},(AA\inv)^2\cap A^2A^{\nv 2}\cap (A\inv A)^2\cap A^{\nv 2}A^2,\langle A\rangle) & \text{in case $(2)$.}
\end{cases}
\]
By increasing $m$ if necessary, we may assume $W\seq V^m$ without loss of generality.

Assume $\Sigma$ has exponent $r$. By Lemma \ref{lem:MWKP}, there is a nonempty finite symmetric set $Y\seq G$ such that $Y^4\seq W$ and $V^m$ is covered by $O_{k,m}(1)$ left translates of $Y$. Since $Y^2\seq Y^4\seq W\seq V^m$, it follows that $Y$ is an $O_{k,m}(1)$-approximate group. By Theorem \ref{thm:BGTbdd}, $Y^4$ contains a subgroup $H\leq G$ such that $Y$ is covered by $O_{k,m,r}(1)$ left cosets of $H$. Then $H\seq W$ and $V^m$ is covered by $O_{k,m,r}(1)$ left cosets of $H$. This proves part $(a)$ in both cases $(1)$ and $(2)$. 

For part $(b)$, suppose $\Sigma=V^m$. By part $(a)$ there is a subgroup $K\leq \Sigma$, of index $O_{k,m,r}(1)$, such that $K\seq W$. If $H=\bigcap_{g\in \Sigma}gKg\inv$, then $H$ is normal in $\Sigma$, $H\seq W$, and $[G:H]\leq [G:K]!\leq O_{k,m,r}(1)$. 
\end{proof}

\section{Saturated extensions and approximate Bohr neighborhoods}\label{sec:G*}

Throughout this section, let $G$ be an ultraproduct constructed as in Section \ref{sec:G}. We will now endow  $G$ with a first-order structure, and then take a sufficiently saturated elementary extension $G_*$. Specifically, we define the \textbf{internal language of $G$}, denoted $\cL$, to be the group language together with a unary relation $R_X$ for any internal $X\seq G$. We view $G$ as an $\cL$-structure by interpreting each $R_X$ as $X$. We also view each $G_s$ as an $\cL$-structure by interpreting $R_X$ as some set $X(G_S)\seq G_s$, so that $X=\prod_{\cU}X(G_s)$. In particular, $G$ is also the ultraproduct of the sequence of \emph{$\cL$-structures} $(G_s)_{s\in\N}$.

Now let $G_*$ be a sufficiently saturated elementary extension of $G$ with respect to the language $\cL$.\footnote{``Sufficiently saturated" typically means $\kappa$-saturated and strongly $\kappa$-homogeneous for some very large (e.g. strongly inaccessible) cardinal $\kappa$.} When we say $X\seq G_*$ (resp. $X\seq G$) is \emph{definable}, we mean definable in the language $\cL$ using parameters from $G_*$ (resp. from $G$). If we want to specify that $X$ is definable using parameters from some set $C$, we will say \emph{$C$-definable}.  Let $A_*$ be the interpretation in $G_*$ of the predicate in $\cL$ naming $A$.

Note that the measure $\mu$ naturally extends to $G$-definable subsets of $G_*$. In particular, given a $G$-definable set $X\seq G_*$, the interpretation $X(G)$ of $X$ in $G$ is internal, and so we let $\mu(X)=\mu(X(G))$. We say that a $G$-definable set $X\seq G_*$ is \textbf{pseudofinite} if $X(G)$ is an ultraproduct of finite sets.

\begin{remark}
Although it will not be  necessary for our work, we recall that $\mu$ can be extended (not necessarily uniquely) to all definable subsets of $G_*$. For example, one can add a sort for $[0,1]$ and a function $f_\phi$ for each formula $\phi(x;\ybar)$, from the home sort to $[0,1]$, which is interpreted as $f_\phi(\bbar)=\mu(\phi(G;\bbar))$. Then take $G_*$ to be a saturated extension in this larger language.  See \cite[Section 2]{HPP}. 
\end{remark} 

A cardinal is \textbf{bounded} if it is strictly smaller than the saturation of $G_*$. A set $X\seq G_*$ is \textbf{type-definable} (resp. \textbf{countably type-definable}) if it is an intersection of a bounded (resp. countable) number of definable subsets of $G_*$.

Now suppose $\Sigma$ is a definable subgroup of $G_*$, and  $\Gamma$ is a type-definable normal subgroup of $\Sigma$ such that $[\Sigma:\Gamma]$ is bounded. Call a set $X\seq \Sigma/\Gamma$  \textbf{closed} if $\pi\inv(X)$ is type-definable, where $\pi$ is the canonical projection from $\Sigma$. It is a standard fact that this defines a topology on $\Sigma/\Gamma$, called the \textbf{logic topology}, under which $\Sigma/\Gamma$ is a compact (Hausdorff) topological group. If  $\Gamma$ is countably type-definable, then $\Sigma/\Gamma$ is second countable. See \cite[Section 2]{PilCLG} for details.

The rest of this section summarizes some tools from \cite{CPTNIP} concerning Bohr neighborhoods in $G_*$ and issues regarding their transfer to $G$ and the groups $G_s$.

Given a compact space $\cX$, we say that a map $f\colon G\to \cX$ is \textbf{definable} if $f\inv(C)$ is type-definable for any closed $C\seq \cX$.
The next proposition is a special case of \cite[Proposition 5.1]{CPTNIP}, and crucially relies on the result of Pillay \cite{PiRCP} that the connected component of a definable compactification of a pseudofinite group is abelian.

\begin{proposition}\label{prop:Bohr}
Suppose $\Sigma$ is a $G$-definable pseudofinite subgroup of $G_*$, and $\Gamma\leq \Sigma$ is a countably type-definable bounded-index normal subgroup of $\Sigma$. Then there is a decreasing sequence $(X_i)_{i=0}^\infty$ of definable subsets of $\Sigma$ such that $\Gamma=\bigcap_{i=0}^\infty X_i$ and, for all $i\in\N$, there are:
\begin{enumerate}[\hspace{5pt}$\ast$]
\item a definable finite-index normal subgroup $H_i\leq \Sigma$, and
\item a definable homomorphism $\pi_i\colon H_i\to \T^{n_i}$, for some $n_i\in\N$, 
\end{enumerate}
such that $\Gamma\seq \ker\pi_i\seq X_i\seq H_i$. If, moreover, $\Sigma/\Gamma$ is abelian, then we may assume $H_i=\Sigma$ for all $i\in\N$. 
\end{proposition}

In the setting of the previous proposition, the fact that $X_i$ is definable and contains $\ker\pi_i$ implies that it  contains a Bohr neighborhood $B^{n_i}_{\pi_i,\epsilon_i}$ for sufficiently small $\epsilon_i>0$. However, these Bohr neighborhoods are not necessarily definable, and so we will need to approximate them by definable objects.

\begin{definition}\label{def:approx}
Fix a group $H$ and an integer $n\in\N$.
\begin{enumerate}
\item Given $\delta>0$, we say that a function $f\colon H\to\T^n$ is a \textbf{$\delta$-homomorphism} if $f(1)=0$ and, for all $x,y\in H$, $d_n(f(xy),f(x)+f(y))<\delta$.
\item Given $\delta,\epsilon>0$, we say that $Y\seq H$ is a \textbf{$\delta$-approximate $(\epsilon,n)$-Bohr neighborhood in $H$} if there is a $\delta$-homomorphism $f\colon H\to \T^n$ such that $Y=\{x\in H:d_n(f(x),0)<\epsilon\}$. 
\item Assume $H$ is a definable subgroup of $G_*$, and $\pi\colon H\to \T^n$ is a definable homomorphism. Given an integer $t\geq 1$, we say that a decreasing sequence $(Y_i)_{i=0}^\infty$ of subsets of $H$ is a \textbf{definable $(t,\pi)$-approximate Bohr chain in $H$} if $\bigcap_{i=0}^\infty Y_i=\ker\pi$ and there is a decreasing sequence $(\delta_i)_{i=0}^\infty$ in $\R_{>0}$ converging to $0$ such that, for all $i\geq 0$, $Y_i=\{x\in H:d_n(f_i(x),0)<t\delta_i\}$ for some definable $\delta_i$-homomorphism $f_i\colon H\to \T^n$ with finite image. 
\end{enumerate}
\end{definition}

Note that if $(Y_i)_{i=0}^\infty$ is a definable $(t,\pi)$-approximate Bohr chain in $H$, then each $Y_i$ is a $\delta_i$-approximate $(t\delta_i,n)$-Bohr set in $H$. 
It is also worth emphasizing that each $Y_i$ is indeed a \emph{definable} subset of $H$ (see  \cite[Proposition 5.3]{CPTNIP}). The next result is a special case of \cite[Lemma 5.4]{CPTNIP}.

\begin{lemma}\label{lem:Bapprox}
Suppose $H\leq G_*$ is definable and $\pi\colon H\to \T^n$ is a definable homomorphism for some $n\in\N$. Then, for any integer $t\geq 1$, there is a definable $(t,\pi)$-approximate Bohr chain $(Y_i)_{i=0}^\infty$ in $H$.
\end{lemma}

Finally, we state a special case of \cite[Corollary 4.4]{CPTNIP}, which is an immediate consequence of \cite[Theorem 5.13]{AlGlGo}.

\begin{proposition}\label{prop:findBohr}
There is a real number $\theta>0$ such that if $H$ is a finite group, $n\in\N$, and $0<\delta<\theta$, then every $\delta$-approximate $(3\delta,n)$-Bohr neighborhood in $H$ contains a $(\delta,n)$-Bohr neighborhood in $H$.
\end{proposition}

\section{Proof of Theorem \ref{thm:maingen}}\label{sec:genproof}

\subsection{Transfer to $G_*$} Throughout this subsection, let $G$ be an ultraproduct constructed as in Section \ref{sec:G}, and let $G_*$ be the saturated extension from Section \ref{sec:G*}. The goal of this subsection is to transfer the analysis in Section \ref{sec:CSS} to the saturated group $G_*$. The main idea is that the decreasing sequence $(X_n)_{n=0}^\infty$ of internal sets constructed in Corollary \ref{cor:MWKP} converges to a subgroup of $G$, which is ``large" in a certain sense. By transferring the sequence first to $G_*$, we will have more precise control over exactly what this means, and it will be easier to find normal subgroups.  

Recall that $A_*$ is the interpretation in $G_*$ of our distinguished internal set $A\seq G$.

\begin{lemma}\label{lem:Gamma}$~$
\begin{enumerate}[$(a)$]
\item Suppose $\mu(A_*A\inv_* A_*)<\infty$. Then there is a countably type-definable subgroup $\Gamma\leq G_*$ such that $\Gamma\seq (A_*A_*\inv)^2$ and $\Gamma$ has index at most $2^{\aleph_0}$ in $\langle A_* A_*\inv\rangle$.
\item Suppose $\mu(A_*^3)<\infty$. Then there is a countably type-definable subgroup $\Gamma\leq G_*$ such that $\Gamma\seq (A_*A_*\inv)^2\cap A_*^2A_*^{\nv 2}\cap (A_*\inv A_*)^2\cap A_*^{\nv 2}A_*^2$ and $\Gamma$ has index at most $2^{\aleph_0}$ in $\langle  A_*\rangle$.
\end{enumerate}
\end{lemma}
\begin{proof}
Set
\[
(V,W,\Sigma)=\begin{cases}
(A A\inv)^2,(A A\inv)^2,\langle A A\inv\rangle) & \text{in case $(a)$}\\
(\bar{A}^4,(AA\inv)^2\cap A^2A^{\nv 2}\cap (A\inv A)^2\cap A^{\nv 2}A^2,\langle A\rangle) & \text{in case $(b)$.}
\end{cases}
\]
Let $(V_*,W_*,\Sigma_*)$ be similarly defined, but with $A_*$ in place of $A$. 

Working first in $G$, we apply Corollary \ref{cor:MWKP} to find a sequence $(Y_n)_{n=0}^\infty$ of symmetric, internal subsets of $G$ such that $Y_0\seq W$ and, for any $n\in\N$, $Y^2_{n+1}\seq Y_n$ and $V^2$ is covered by finitely many $\Sigma$-translates of $Y_n$. So, for any $n\in\N$, there is some $k_n\in\N$ such that $V^2$ is covered by finitely many $V^{k_n}$-translates of $Y_n$.

Now in $G_*$, let $X_n$ be the $\emptyset$-definable set given by the interpretation  of the unary relation $R_{Y_n}$. By elementarity,  $X_0\seq W_*$ and, for any $n\in\N$, $X_n$ is symmetric and internal, $X^2_{n+1}\seq X_n$, and $V_*^2$ is covered by finitely many $V_*^{k_n}$-translates of $X_n$. 

Fix $n\in\N$, and let $F\seq \Sigma_*$ be finite such that $V_*^2\seq FX_n$. By induction on $k\geq 1$, we show that $V_*^k\seq F^kX_n$. The base case is given, so assume the result for $k\geq 1$. Then $V_*^{k+1}=V^k_*V_*\seq F^kX_nV_*\seq F^kW_*V_*\seq F^kV^2_*\seq F^{k+1}X_n$.

We have shown that, for any $n\in\N$, there is a countable set $F_n\seq \Sigma_*$ such that $\Sigma_*=FX_n$. Let $\Gamma=\bigcap_{n=0}^\infty X_n$, and note that $\Gamma$ is a countably type-definable subgroup of $G_*$, which is contained in $W_*$. Since $\Sigma_*$ is covered by countably many $\Sigma_*$-translates of $X_n$ for all $n\geq 1$, it follows that $\Gamma$ has index at most $2^{\aleph_0}$ in $\Sigma_*$. 
\end{proof}

\begin{corollary}\label{cor:Gamma}$~$
\begin{enumerate}[$(a)$]
\item Suppose $\mu(A_*A\inv_* A_*)<\infty$. Then there is a countably type-definable subgroup $\Gamma\leq G_*$ such that:
\begin{enumerate}[$(i)$]
\item $\Gamma\seq (A_* A_*\inv)^2$, 
\item $\Gamma$ is normal in $\langle A_*A_*\inv\rangle$, and 
\item $\Gamma$ has index at most $2^{\aleph_0}$ in $\langle A_* A_*\inv\rangle$.
\end{enumerate}
\item Suppose $\mu(A_*^3)<\infty$. Then there is a countably type-definable subgroup $\Gamma\leq G_*$ such that:
\begin{enumerate}[$(i)$]
\item $\Gamma\seq (A_*A_*\inv)^2\cap A_*^2 A_*^{\nv 2}\cap (A_*\inv A_*)^2\cap A_*^{\nv 2}A_*^2$, 
\item $\Gamma$ is normal in $\langle A_*\rangle$, and 
\item $\Gamma$ has index at most $2^{\aleph_0}$ in $\langle A_*\rangle$.
\end{enumerate}
\end{enumerate}
\end{corollary}
\begin{proof}
Set
\[
(V,W,\Sigma)=\begin{cases}
A_* A_*\inv,(A_* A_*\inv)^2,\langle A_* A_*\inv\rangle) & \text{in case $(a)$}\\
(\bar{A}_*,(A_*A_*\inv)^2\cap A_*^2A_*^{\nv 2}\cap (A_*\inv A_*)^2\cap A_*^{\nv 2}A_*^2,\langle A_*\rangle) & \text{in case $(b)$.}
\end{cases}
\]
By Lemma \ref{lem:Gamma}, we have a countably type-definable subgroup $\Gamma_0\leq G_*$ such that $\Gamma_0\seq W$ and $\Gamma_0$ has index at most $2^{\aleph_0}$ in $\Sigma$. Let $\Gamma=\bigcap_{g\in\Sigma}g\Gamma_0g\inv$. Then $\Gamma$ is an intersection of at most $2^{\aleph_0}$ conjugates $g\Gamma_0g\inv$ with $g\in\Sigma$. So $\Gamma$ is a type-definable subgroup of $G_*$, which is normal in $\Sigma$ and has index at most $2^{2^{\aleph_0}}$ in $\Sigma$. 

We now show that $\Gamma$ is countably type-definable of index at most $2^{\aleph_0}$ in $\Sigma$. For the first part, let $\cL_0\seq\cL$ be a countable language containing the language of groups, and unary predicates defining $A$ and $Y_n$ for $n\in\N$, where $Y_n$ are the predicates used to obtain $\Gamma_0$ (via the proof of Lemma \ref{lem:Gamma}). Then $\Gamma$ is $\cL_0$-type-definable. Moreover,   $\Gamma_0$ is $\cL_0$-type-definable over $\emptyset$ and, so $\sigma(\Gamma_0)=\Gamma_0$ for any $\sigma\in \Aut_{\cL_0}(G_*)$. Since $\Sigma$ is $\Aut_{\cL_0}(G_*)$-invariant, $\sigma(\Gamma)=\Gamma$ for any $\sigma\in\Aut_{\cL_0}(G_*)$, and so $\Gamma$ is $\cL_0$-type-definable over $\emptyset$. Since $\cL_0$ is countable, $\Gamma$ is countably type-definable. 

Finally, let $\Gamma=\bigcap_{n=0}^\infty D_n$, where each $D_n$ is definable and (without loss of generality) contained in $\Sigma$. Since $\Gamma$ has bounded index in $\Sigma$, we may fix some bounded set $C\subset\Sigma$ such that $\Sigma =C\Gamma$. Fix $m,n\in\N$. Then  $V^m\seq\Sigma= C\Gamma=CD_n$. By saturation of $G_*$, it follows that there is some finite $C_{n,m}\seq C$ such that $V^m\seq C_{n,m}D_n$. So, if $C_n=\bigcup_{m\in\N}C_{n,m}$, then $C_n$ is countable and $\Sigma = C_nD_n$. Once again, this implies that $\Gamma$ has index at most $2^{\aleph_0}$ in $\Sigma$.
\end{proof}

Corollary \ref{cor:Gamma} is  a nonstandard Bogolyubov-Ruzsa-type statement about pseudofinite sets of small alternation or small tripling. However, since the subgroup $\Gamma$ is not necessarily definable, it cannot be directly transferred to statements about internal subsets of $G$ (which are needed in order to transfer to the finite groups $G_s$). For this, we need the material in Section \ref{sec:G*} on approximate Bohr neighborhoods.

\subsection{Ultraproduct argument}

We now prove parts $(1)$ and $(2)$ of Theorem \ref{thm:maingen} simultaneously.  Given a group $G$ and a set $A\seq G$, let 
\[
\Sigma(A) =\begin{cases}
\langle AA\inv\rangle \\
\langle A\rangle 
\end{cases}
U(A) =\begin{cases}
AA\inv A \\
A^3 
\end{cases}
V(A) =\begin{cases}
AA\inv & \text{\hspace{10pt}in part $(1)$}\\
\bar{A} & \text{\hspace{10pt}in part $(2)$,}
\end{cases}
\]
\[
\text{and }W(A) =\begin{cases}
(AA\inv)^2 & \text{in part $(1)$}\\
(AA\inv)^2\cap A^2A^{\nv 2}\cap (A\inv A)^2\cap A^{\nv 2}A^2 & \text{in  part $(2)$.}
\end{cases}
\]
The ambient group $G$ is supressed from the notation, but this should cause no confusion in the following proof.

The next result is a restatement of Theorem \ref{thm:maingen}, which we will prove by taking an ultraproduct of counterexamples, and using the material in Section \ref{sec:G*} in order to transfer Bohr neighborhoods through ultraproducts and saturated extensions. 

\begin{theorem}\label{thm:gen2}
For any positive integers $k$ and $m$, there is an integer $s=s(k,m)$ such that the following holds. Suppose $G$ is a group and $A\seq G$ is finite such that $|U(A)|\leq k|A|$ and $\Sigma(A)=V(A)^m$. Then there are:
\begin{enumerate}[\hspace{5pt}$\ast$]
\item a normal subgroup $H\leq \Sigma(A)$, of index at most $s$, and  
\item a $(\delta,n)$-Bohr neighborhood $B$ in $H$, with $\delta\inv,n\leq s$,
\end{enumerate}
such that $B\seq W(A)$. Moreover, if $\Sigma(A)$ is abelian then we may assume $H=\Sigma(A)$.
\end{theorem}
\begin{proof}
Suppose not. Then for any $s\in\N$, we may fix a group $G_s$ and a finite set $A_s\seq G_s$ such that $|U(A_s)|\leq k|A_s|$, $\Sigma(A_s)=V(A_s)^m$, and there does not exist a normal subgroup $H\leq \Sigma(A_s)$ and a $(\delta,n)$-Bohr neighborhood $B$ in $H$ such that $[\Sigma(A_s):H],\delta\inv,n\leq s$ and $H\seq W(A_s)$. Note that $|\Sigma(A_s)|>\max\{s,m\}$, since otherwise we could take $H=\{1\}$.

Let $\cU$ be a nonprincipal ultrafilter on $\N$ and set $G=\prod_{\cU}G_s$. Let $A=\prod_{\cU}A_s$, and note that $A$ is an internal subset of $G$. We also have $U(A)=\prod_{\cU}U(A_s)$, $V(A)=\prod_{\cU}V(A_s)$, $W(A)=\prod_{\cU}W(A_s)$, and
\[
\Sigma(A)=\bigcup_{n\in\N}V(A)^n=\bigcup_{n\in\N}\prod_{\cU}V(A_s)^n=\prod_{\cU}V(A_s)^m= V(A)^m
\]
Note, in particular, that $\Sigma(A)=\prod_{\cU}\Sigma(A_s)$ is infinite, and so $G$ is infinite. Let $\mu$ be the $|A|$-normalized pseudofinite counting measure on internal subsets of $G$. By {\L}o{\'s}'s Theorem, $\mu(U(A))\leq k<\infty$.

Let $G_*$ be a saturated elementary extension of $G$ in the internal language of $G$ (see Section \ref{sec:G*}), and let $A_*$ be the interpretation in $G_*$ of the predicate defining $A$ in $G$. By Lemma \ref{lem:Gamma}, there is a countably type-definable subgroup $\Gamma\leq G_*$ such that $\Gamma\seq W(A_*)$ and $\Gamma$ has index at most $2^{\aleph_0}$ in $\Sigma(A_*)$. Note also that $\Sigma(A_*)=V(A_*)^m$. In particular, $\Sigma(A_*)$ is $G$-definable and pseudofinite. 

By Proposition \ref{prop:Bohr} and saturation of $G_*$, there is a definable finite-index normal subgroup $H\leq\Sigma(A_*)$ and a definable homomorphism $\pi\colon H\to\T^n$, for some $n\in\N$,  such that $\Gamma\seq\ker\pi\seq H\cap W(A_*)$. By Lemma \ref{lem:Bapprox}, there is a definable $(3,\pi)$-approximate Bohr chain $(Y_i)_{i=0}^\infty$ in $H$. By saturation, $Y_i\seq W(A_*)$ for sufficiently large $i\in\N$.  So we may fix $\delta<\theta$, where $\theta$ is as in Proposition \ref{prop:findBohr}, and a definable $\delta$-homomorphism $f\colon H\to \T^n$ such that $Y:=\{x\in H:d(f(x),0)<3\delta\}\seq W(A^*)$.

 Let $\Lambda=f(H)$, and note that $\Lambda$ is finite (see Definition \ref{def:approx}$(3)$). Given $\lambda\in\Lambda$, let $F(\lambda)=f\inv(\lambda)\seq H$. Then each $F(\lambda)$ is definable. Set $r=[\Sigma(A_*):H]<\infty$. 

Fix $\cL$-formulas $\phi(x;\ybar)$, $\psi(x;\zbar)$, and $\zeta_\lambda(x;\ubar_\lambda)$ for $\lambda\in \Lambda$, such that $H$ is defined by an instance of $\phi(x;\ybar)$, $Y$ is defined by an instance of $\psi(x;\zbar)$, and, for $\lambda\in\Lambda$, $F(\lambda)$ is defined by an instance of $\zeta_\lambda(x;\ubar_\lambda)$. Let $I\seq\N$ be the set of $s\in\N$ such that, for some tuples  $\abar_s$, $\bbar_s$, and $\cbar_{\lambda,s}$ (for $\lambda\in\Lambda$) from $G_s$, we have:
\begin{enumerate}[$(i)$]
\item $\phi(x;\abar_s)$ defines a normal subgroup $H_s$ of $\Sigma(A_s)=V(A_s)^m$ of index $r$,
\item for all $\lambda\in\Lambda$, $\zeta_\lambda(x;\cbar_{\lambda,s})$ defines a subset $F_s(\lambda)$ of $H_s$,
\item if $f_s\colon H_s\to \Lambda$ is defined so that $f_s(x)=\lambda$ if and only if $x\in F_s(\lambda)$, then $f_s$ is a well-defined $\delta$-homomorphism (from $H_s$ to $\T^n$),
\item $\psi(x;\bbar_s)$ defines a subset $Y_s$ of $H_s$, and $Y_s=\{x\in H_s:d(f_s(x),0)<3\delta\}$,
\item $Y_s\seq W(A_s)$.
\end{enumerate}
Then $I\in\cU$ by {\L}o\'{s}'s Theorem and elementarity (checking that $(i)$ through $(v)$ are first-order expressible is somewhat cumbersome, but fairly routine; see the proof of \cite[Lemma 5.6]{CPTNIP}). So we may fix some $s\in I$ such that $r,n,\delta\inv\leq s$. For this $s$, $Y_s$ is a $\delta$-approximate $(3\delta,n)$-Bohr set in $H_s$. By Proposition \ref{prop:findBohr}, there is a $(\delta,n)$-Bohr set $B\seq Y_s$. So $B\seq W(A_s)$, which contradicts the choice of $G_s$ and $A_s$. 

Finally, if we assume $\Sigma(A)$ is abelian then, in the above proof, we may take $H=\Sigma(A_*)$ by Proposition \ref{prop:Bohr}, and thus assume $H_s=\Sigma(A_s)$ for all $s\in\N$.
\end{proof}

\begin{remark}\label{rem:abelianBohr}
Suppose that in Theorem \ref{thm:gen2} we further assume  $G$ is abelian and $|A|\geq c|G|$ for some fixed $c>0$. Then we have $\langle A\rangle=\bar{A}^m$, where $m\leq \lceil3c\inv+1\rceil$), and $[G:\langle A\rangle]\leq \lceil c\inv\rceil$. Therefore, in the proof of the theorem, $\Sigma(A_*)$ has finite index in $G_*$, and so $\Gamma$ has index at most $2^{\aleph_0}$ in $G_*$. So we can carry out the rest of the proof with $G_*$ in place of $\Sigma(A_*)$, obtaining $H_s=G_s$ in the conclusion. Consequently, in Theorem \ref{thm:Bogogen}$(b)$, if $G$ is abelian then we may take $H=G$. 
\end{remark}

\section{Arithmetic regularity and VC-dimension}\label{sec:NIP}

The goal of this section is to prove Theorem \ref{thm:NIPregexp}. As indicated in the introduction, the only ingredient in the work of Alon, Fox, and Zhao \cite{AFZ} requiring abelian groups is Theorem \ref{thm:Bogo}$(a)$. The (qualitative) nonabelian version of this result for sets of small \emph{tripling}, provided by Corollary \ref{cor:sym}$(a)$, will be sufficient to essentially carry out the same proof as in \cite{AFZ} (see also Remark \ref{rem:Tao}). The only extra work is in specifying the numerics and clarifying the ``regularity" aspect the result (i.e. condition $(ii)$ of Theorem \ref{thm:NIPregexp}). We also make some similar (but mostly qualitative) statements for purely nonabelian finite groups, and finite simple groups. 

\begin{definition}
Let $G$ be a finite group. Given a subset $A\seq G$ and some $\epsilon>0$, define the \textbf{$\epsilon$-stabilizer of $A$} to be the set $\Stab_\epsilon(A):=\{x\in G:|xA\smd A|\leq\epsilon|G|\}$.
\end{definition}

The following lemma, which we have extracted  from the counting techniques done in \cite{AFZ}, makes explicit the connection between $\epsilon$-stabilizers and strong arithmetic regularity involving subgroups.

\begin{lemma}\label{lem:separate}
Let $G$ be a finite group and fix a subset $A\seq G$ and some $\epsilon>0$. Suppose $H$ is a subgroup of $G$ contained in $\Stab_\epsilon(A)$. 
\begin{enumerate}[$(a)$]
\item There is $D\seq G$, which is a union of right cosets of $H$, such that $|A\smd D|\leq \epsilon|G|$.
\item There is $Z\seq G$, with $|Z|<\frac{1}{2}\epsilon^{1/2}|G|$, such that for any $x\in G\backslash Z$, either $|Hx\cap A|\leq\epsilon^{1/4}|H|$ or $|Hx\backslash A|\leq\epsilon^{1/4}|H|$.
\end{enumerate}
\end{lemma}
\begin{proof}
Let $\cC$ be the set of right cosets of $H$ in $G$. Given $C\in \cC$, define $P_C=(C\cap A)\times (C\backslash A)$. Let $P=\bigcup_{C\in \cC}P_C=\{(a,g)\in A\times G\backslash A:ga\inv\in H\}$, and note that $P_C\cap P_{C'}=\emptyset$ for distinct $C,C'\in\cC$. From the proof of \cite[Lemma 2.4]{AFZ}, one obtains
\begin{equation*}
2\sum_{C\in\cC}|P_C|=2|P|=\sum_{x\in H}|xA\smd A|\leq\epsilon|G||H|.\tag{$\dagger$}
\end{equation*} 
For part $(a)$, we continue to follow \cite{AFZ}. Let $D=\bigcup\{C\in\cC:|C\cap A|\geq|H|/2\}$. Then, by $(\dagger)$,
\[
|A\smd D| =\sum_{C\in\cC}\min\{|C\cap A|,|C|-|C\cap A|\}\leq \sum_{C\in\cC}\frac{2}{|H|}|P_C|\leq \epsilon|G|.
\]
For part $(b)$, let $\cZ=\{C\in\cC:|P_C|>\epsilon^{1/2}|H|^2\}$. By $(\dagger)$,
\[
\textstyle\frac{1}{2}\epsilon|G||H|\geq \displaystyle \sum_{C\in\cC}|P_C|>\epsilon^{1/2}|H|^2|\cZ|.
\]
So $|\cZ|<\frac{1}{2}\epsilon^{1/2}\frac{|G|}{|H|}$. Now set $Z=\bigcup_{C\in\cZ}C$. Then $|Z|<\frac{1}{2}\epsilon^{1/2}|G|$. Moreover, if $x\in G\backslash Z$ then $Hx\not\in\cZ$, and so $|P_{Hx}|\leq \epsilon^{1/2}|H|^2$, which implies $|Hx\cap A|\leq\epsilon^{1/4}|H|$ or $|Hx\backslash A|\leq\epsilon^{1/4}|H|$.
\end{proof}

We can now prove Theorem \ref{thm:NIPregexp}, following the same steps as in \cite{AFZ}.

\begin{proof}[Proof of Theorem \ref{thm:NIPregexp}]
Fix positive integers $r$ and $d$, and real numbers $\epsilon, \nu>0$. Suppose $G$ is a finite group of exponent at most $r$, and $A\seq G$ has VC-dimension at most $d$. 
Let $S=\Stab_\delta(A)$, where $\delta=(\epsilon/4)^{(d+ \nu)/d}/30^{ \nu/d}$. Note that $S$ is symmetric.  Set $k=(30/\delta)^d$ and $p=d(d+ \nu)/ \nu$. It is an immediate consequence of \emph{Haussler's Packing Lemma} \cite{HaussPL}, for sets systems of finite VC-dimension, that   $|S|\geq|G|/k$ (see Lemmas 2.1 and 2.2 of \cite{AFZ}). Therefore, we cannot have $|S^{3^{i+1}}|>3^{p}|S^{3^i}|$ for all $i\leq \log_{3^{p}}(k)$. So we may fix some $t\leq \log_{3^{p}}(k)$ such that, setting $B=S^{3^t}$, we have  $|B^3|\leq 3^{p}|B|$. By Corollary \ref{cor:sym}$(a)$, there is a subgroup $H\leq G$ such that $H\seq B^4$ and $B$ is covered by $O_{r,d, \nu}(1)$ left translates of $H$. Since $|G|\leq k|B|$, we see that $H$ has index at most $O_{r,d, \nu}(k)=O_{r,d, \nu}((1/\epsilon)^{d+ \nu})$. 

To finish the proof, it suffices by Lemma \ref{lem:separate} to show that $H\seq\Stab_\epsilon(A)$. We have $|xA\smd A|\leq \delta |G|$ for all $x\in S$, and $H\seq B^4=S^{4\cdot 3^{t}}$. So, for any $x\in H$,  
\[
|xA\smd A|\leq 4\cdot 3^t\delta|G|\leq 4 k^{1/p}\delta|G|=\epsilon|G|.\qedhere
\]
\end{proof}

\begin{remark}\label{rem:NIPreg}
We make some comments to follow up on Remark \ref{rem:NIPpre}.
\begin{enumerate}[\hspace{0pt}$(1)$]
\item Note that, in the proof of Theorem \ref{thm:NIPregexp}, if $K=\bigcap_{g\in G}gHg\inv$ then $K$ is normal of index at most $[G:H]!$ and $K\seq \Stab_\epsilon(A)$. So, if $[G:H]\leq O_{r,d, \nu}((1/\epsilon)^{d+ \nu})$, for some chosen $\epsilon, \nu>0$, then  $\log [G:K]\leq O_{r,d, \nu}(\epsilon^{\nv(d+ \nu)}\log(\epsilon^{\nv 1}))$. Altogether, we have a statement identical to Theorem \ref{thm:NIPregexp}, but with a \emph{normal} subgroup of index $2^{O_{r,d, \nu}((1/\epsilon)^{d+ \nu})}$. One reason a normal subgroup is desirable in this situation is that it implies a very strong graph regularity conclusion for the bipartite graph $xy\in A$ on $G$, in which the pieces of the regular partition are the cosets of $H$ (see \cite[Corollary 3.3]{CPTNIP}). 
\item A non-effective version of Theorem \ref{thm:NIPregexp}, with a normal subgroup, can also be proved by applying Corollary \ref{cor:sym}$(b)$  directly to $\Stab_\epsilon(A)$. Together with Haussler's Packing Lemma, this would directly yield a normal subgroup of index $O_{r,d,\epsilon}(1)$ contained in $\Stab_\epsilon(A)$. It is interesting to note that a qualitative version of Theorem \ref{thm:NIPregexp}, with a normal subgroup, was already shown in \cite{CPTNIP} using fairly different techniques (although there are some aspects of the work in \cite{CPTNIP} which are not recovered here, including definability of the subgroup $H$ and stronger regularity statement). 
\end{enumerate}
\end{remark}

Finally, we prove similar results about purely nonabelian finite groups and finite simple groups. To motivate our interest in this setting, we recall some  of the previous work on arithmetic regularity for subsets of finite groups satisfying extra tameness properties. One example of such a property is bounded VC-dimension, which we have already discussed. Another important example is that of a \textbf{$d$-stable} subset $A$ of a group $G$, for some integer $d\geq 1$, which means there do not exist $a_1,\ldots,a_d,b_1,\ldots,b_d\in G$ such that $a_ib_j\in A$ if and only if $i\leq j$. Note that a $d$-stable set  has VC-dimension at most $d-1$.  Both of these properties were previously studied in the setting of Szemer{\'e}di regularity for graphs (see \cite{AFN}, \cite{MaShStab}).

In \cite{CPT} (joint with Pillay and Terry), we showed that, given $d\geq 1$ and $\epsilon>0$, if $G$ is a finite group and $A\seq G$ is $d$-stable then there is a normal subgroup $H\leq G$, of index $O_{d,\epsilon}(1)$, and a union $D$ of cosets of $G$ such that $|A\smd D|\leq \epsilon|H|$. Informally, stable subsets of finite groups are structurally approximated by cosets of a bounded-index normal subgroup. In the setting of finite groups, this phenomenon was first investigated by Terry and Wolf \cite{TeWo}, who proved a similar result for $G=\F_p^n$ with strong quantitative bounds, but with the approximation $|A\smd D|\leq \epsilon|G|$. (This was recently generalized to arbitrary finite abelian groups in \cite{TeWo2}.) 

In contrast, easy examples show that subgroups are not sufficient to control sets of bounded VC-dimension. For example, as noted in \cite{AFZ}, if $p\geq 3$ is prime and $G=\Z/p\Z$ and $A=\{1,\ldots,\frac{p-1}{2}\}$, then $A$ has VC-dimension $2$, but $A$ cannot be approximated by cosets of a nontrivial subgroup of $G$. This is one reason to use Bohr neighborhoods in the  formulation of arithmetic regularity for sets of bounded VC-dimension, which was done by Sisask in the abelian setting \cite{SisNIP}, and independently in \cite{CPTNIP} for general finite groups. As we have seen above, if one introduces a uniform bound on the exponent of the groups, then subgroups can be used to approximate sets of bounded VC-dimension. So this motivates the following result that purely nonabelian groups (see Corollary \ref{cor:pure}) also exhibit this behavior.

\begin{theorem}\label{thm:NIPregpna}
Fix a positive integer $d$. Suppose $G$ is a purely nonabelian finite group, and $A\seq G$ has VC-dimension at most $d$. Then, for any $\epsilon>0$, there is a normal subgroup $H$ of $G$, of index $O_{d,\epsilon}(1)$, which satisfies the following properties.
\begin{enumerate}[$(i)$]
\item \textnormal{(structure)} There is a set $D\seq G$, which is a union of cosets of $H$, such that $|A\smd D|\leq \epsilon|G|$.
\item \textnormal{(regularity)} There is a set $Z\seq G$, with $|Z|<\frac{1}{2}\epsilon^{1/2}|G|$, such that for any $x\in G\backslash Z$, either $|xH\cap A|\leq\epsilon^{1/4}|H|$ or $|xH\backslash A|\leq\epsilon^{1/4}|H|$.
\end{enumerate}
\end{theorem}
\begin{proof}
Fix a purely nonabelian finite group $G$, a subset $A\seq G$ of VC-dimension at most $d$, and $\epsilon>0$. As in Theorem \ref{thm:NIPregexp}, if $S=\Stab_{\epsilon/4}(A)$ then $|S|\geq (\epsilon/120)^d|G|$. By Corollary \ref{cor:pure} there is a normal subgroup $H\leq G$, of index $O_{d,\epsilon}(1)$, such that $H\seq S^4$. So $|xA\smd A|\leq \epsilon|G|$ for any $x\in H$. Now apply Lemma \ref{lem:separate}.
\end{proof}

\begin{remark}
The previous theorem can also be deduced from \cite[Theorem 5.7]{CPTNIP}, yielding further information as discussed in Remark \ref{rem:NIPreg}$(2)$. On the other hand, the proof here seems more direct, and certainly uses a more acute application of VC-theory. (Both proofs involve identical uses of \cite{AlGlGo} and \cite{PiRCP}).  
\end{remark}

The work in \cite{CPT} on stable regularity implies that, for any $d\geq 1$ and $\epsilon>0$, if $G$ is a finite simple group of size $\Omega_{d,\epsilon}(1)$ and $A\seq G$ is $d$-stable, then $|A|\leq\epsilon|G|$ or $|A|\geq (1-\epsilon)|G|$.\footnote{This is a finitary analogue of the older fact that any definable subset of an (infinite) definably-connected stable group has measure $0$ or $1$ with respect to the unique Keisler measure.} For the abelian case (i.e. $G=\Z/p\Z$), a quantitative lower bound on $p=|G|$, in terms of $d$ and $\epsilon$, could be deduced from \cite{TeWo2}. On the other hand, the example above, which shows that subgroups are not sufficient to approximate sets of bounded VC-dimension, takes place in \emph{abelian} finite simple groups. This motivates  the following corollary of Theorem \ref{thm:NIPregpna}.

\begin{corollary}\label{cor:NIPregpna}
For any integer $d$ and any $\epsilon>0$, there is an integer $n=n(d,\epsilon)$ such that, if $G$ is a nonabelian finite simple group of size greater than $n$, and $A\seq G$ has VC-dimension at most $d$, then $|A|\leq\epsilon|G|$ or $|A|\geq (1-\epsilon)|G|$. 
\end{corollary}

\begin{remark}
Using a similar strategy, we can give a direct proof of the previous corollary, which yields $\log (n(d,\epsilon))\leq O((\epsilon/90)^{\nv 6d})$ as an explicit bound. Namely, by the work of Gowers \cite{GowQRG} discussed in Remark \ref{rem:simple}, there is some $c>0$ such that if $G$ is a nonabelian finite simple group with $\log|G|\geq c(\epsilon/90)^{\nv 6d}$, and $S\seq G$ is such that $|S|\geq (\epsilon/90)^{\nv d}|G|$, then $G=S^3$. So fix such a $G$, and suppose $A\seq G$ is of VC-dimension at most $d$.  By Haussler's Packing Lemma, and choice of $c$, we have $G=(\Stab_{\epsilon/3}(A))^3=\Stab_\epsilon(A)$. Now apply Lemma \ref{lem:separate}.\footnote{As  in Remark \ref{rem:simple}, the work in \cite{CollJCLG} implies $n(d,\epsilon)\leq (\lceil (\epsilon/90)^{\nv 3d}\rceil+1)!$ in Corollary \ref{cor:NIPregpna}.}
\end{remark}

\section{Final remarks}\label{sec:final}

\subsection{Quantitative bounds} \label{sec:explicit} An obvious question at this point is on effective bounds for Theorems \ref{thm:Bogogen}, \ref{thm:mainbdd}, and \ref{thm:maingen}. Our proof of Theorem \ref{thm:maingen} used an ultraproduct construction, and did not give explicit bounds of any kind. While ultraproducts do not appear explicitly in our proof of Theorem \ref{thm:mainbdd}, they are similarly used in previous work on approximate groups (both in \cite{BGT} and \cite{HruAG}). 

It is sometimes possible, with enough work, to reverse engineer effective bounds from arguments with ultraproducts, but these bounds are usually very bad (see, e.g., \cite[Chapter 7]{TaoH5P} for some discussion on this topic). Altogether, it seems that in order to obtain efficient bounds for the above results, one would need efficient bounds for results on approximate groups, or different proof strategy altogether. 

\subsection{Small tripling vs. approximate groups} \label{sec:Tao} For the sake of completeness, we note that weaker versions of our main results can be obtained without the revised Sanders-Croot-Sisask analysis in Section \ref{sec:CSS}. This is because of the following result of Tao, which follows from the proof of \cite[Theorem 3.9]{TaoPSE} (or see \cite[Corollary 5.2]{BGT}).

\begin{theorem}[\cite{TaoPSE}]\label{thm:Tao}
Suppose $A$ is a nonempty finite subset of a group $G$. If $|A^3|\leq k|A|$ then $\bar{A}^2$ is an $O(k^{O(1)})$-approximate group containing $A$. 
\end{theorem}

Together with Theorem \ref{thm:BGTbdd}, one obtains a weaker version of Theorem \ref{thm:mainbdd}.

\begin{corollary}[\cite{BGT}, \cite{TaoPSE}]\label{cor:Tao}
Fix positive integers $k$ and $r$. Let $G$ be a group of exponent $r$, and fix a finite subset $A\seq G$. Suppose $|A^3|\leq k|A|$. Then there is $H\leq\langle A\rangle$ such that $\bar{A}^2$ is covered by $O_{k,r}(1)$ left cosets of $H$ and $H\leq \bar{A}^8$. 
\end{corollary}

\begin{remark}\label{rem:Tao}
Corollary \ref{cor:Tao} could be used instead of Corollary \ref{cor:sym}$(a)$ in the proof Theorem \ref{thm:NIPregexp}.
\end{remark}

\begin{remark}\label{rem:HruTao}
Recall that if $A\seq G$ is finite and nonempty, with $|AA\inv A|\leq k|A|$, then $|(A\inv A)^3|\leq k^{O(1)}|A\inv A|$ by Proposition \ref{prop:Ruz}$(b)$, and so $(AA\inv)^2$ is an $O(k^{O(1)})$-approximate group by Theorem \ref{thm:Tao}. Altogether, this is essentially the ``discrete case" of \cite[Corollary 3.11]{HruAG}.  
\end{remark}

A weaker version of Theorem \ref{thm:maingen} can also be formulated using Theorem \ref{thm:Tao}, but the proof would still require our work with saturated extensions and approximate Bohr neighborhoods, and so we will not go into it any further. On the other hand, the following statement about sets of small tripling in arbitrary groups follows by combining Theorem \ref{thm:Tao} with the main structure theorems for approximate groups from Breuillard, Green, and Tao \cite{BGT}.

\begin{theorem}[\cite{BGT}, \cite{TaoPSE}]\label{thm:BGTtrip}
Fix a positive integer $k$. Suppose $G$ is a group and $A\seq G$ is finite and nonempty, with $|A^3|\leq k|A|$. Then there is a subgroup $H$ of $G$ and a finite normal subgroup $N$ of $H$ with the following properties:
\begin{enumerate}[$(i)$]
\item $A$ is covered by $O_k(1)$ left cosets of $H$;
\item $H/N$ is nilpotent and finitely generated of rank and step $O_k(1)$;
\item $\bar{A}^8$ contains $N$ and a generating set for $H$.
\end{enumerate}
Moreover, there is a coset nilprogression $P\seq \bar{A}^8$ of rank and step $O_k(1)$ such that $A$ is covered by $O_k(1)$ left translates of $P$. 
\end{theorem} 

The final clause of the previous theorem is a non-abelian analogue of the \emph{Bogolyubov-Ruzsa Lemma} for finite abelian groups, which was stated after Theorem \ref{thm:Bogo}. However, we have the qualitative discrepancy between $\bar{A}^8$ in condition $(iii)$ and $2A-2A$ in the abelian case. Given our earlier results, one naturally wonders if $\bar{A}^8$ can be replaced by $(AA\inv)^2\cap A^2A^{\nv 2}\cap (A\inv A)^2\cap A^{\nv 2}A^2$. This also raises a similar question about small alternation. So we observe that these issues can be addressed simply by combining the results in \cite{BGT} with Lemma \ref{lem:MWKP}. 

\begin{theorem}\label{thm:BGT+}
Fix a positive integer $k$. Suppose $G$ is a group and $A\seq G$ is finite and nonempty. Furthermore,
\begin{enumerate}[$(a)$]
\item assume $|AA\inv A|\leq k|A|$ and set $V=AA\inv$ and $W=(AA\inv)^2$, or
\item assume $|A^3|\leq k|A|$ and set $V=\bar{A}$ and $W=(AA\inv)^2\cap A^2A^{\nv 2}\cap (A\inv A)^2\cap A^{\nv 2}A^2$.
\end{enumerate}
Then there is a subgroup $H$ of $G$ and a finite normal subgroup $N$ of $H$ with the following properties:
\begin{enumerate}[$(i)$]
\item for all $m\geq 1$, $V^m$ is covered by $O_{k,m}(1)$ left cosets of $H$;
\item $H/N$ is nilpotent and finitely generated of rank and step  $O_{k}(1)$;
\item $W$ contains $N$ and a generating set for $H$.
\end{enumerate}
Moreover, there is a coset nilprogression $P\seq W$ of rank and step $O_{k}(1)$ such that for all $m\geq 1$, $V^m$ is covered by $O_{k,m}(1)$ left translates of $P$. 
\end{theorem}
\begin{proof}
By Lemma \ref{lem:MWKP}, there is a symmetric set $Y\seq G$ such that $Y^8\seq W$ and $V^3$ is covered by $O_k(1)$ left translates of $Y$. Since $Y\seq V^2$, it follows that for all $m\geq 1$, $V^m$ is covered by $O_{k,m}(1)$ left translates of $Y$. Note also that $Y$ is an $O_k(1)$-approximate group. So by \cite[Theorem 1.6]{BGT}, there are $N\trianglelefteq H\leq G$ such that $Y$ is covered by $O_k(1)$ left cosets of $H$, $H/N$ is nilpotent and finitely generated of rank and step  $O_k(1)$, and $Y^4$ contains $N$ and a generating set for $H$. Moreover, by \cite[Theorem 2.10]{BGT}, there is a coset nilprogression $P\seq Y^8$ of rank and step $O_k(1)$ such that $|Y|\leq O_k(|P|)$ and $Y$ is covered by $O_k(1)$ left translates of $P$. Altogether, the result follows by choice of $Y$.
\end{proof}

\bibliographystyle{amsplain}
\end{document}